\documentclass{imanum}
\usepackage{tim,epsfig,algorithm,algorithmic}

\usepackage{cite}
\usepackage{ma580}
\usepackage{chen}
\usepackage{graphicx}
\usepackage{hyperref}
\usepackage{xcolor,mdframed}
\usepackage{url}
\usepackage{doi}

\newtheorem{assumption}{Assumption}[section]

\def\calt{\cal T}

\def\b{{\bf{b}}}
\def\u{{\bf{u}}}
\def\v{{\bf{v}}}
\def\c{{\bf{c}}}
\def\d{{\bf{d}}}

\usepackage{listings}
\lstdefinestyle{code}{
  language=Julia,
  xleftmargin=1.5cm,
  showstringspaces=false,
}
\lstnewenvironment{code}{
\lstset{
  language=Julia,
  xleftmargin=1.5cm,
  showstringspaces=false,
}}{}

\usepackage{amsfonts}
\usepackage{mathrsfs}
\usepackage{lipsum}
\usepackage{geometry}
\usepackage{multirow}
\usepackage{graphics,graphicx,epsf,subfigure,epstopdf,epsfig}
\usepackage{amssymb}
\usepackage{amsmath}
\usepackage{float}
\usepackage{boxedminipage}
\usepackage{caption}
\usepackage{enumerate}
\usepackage{enumitem}
\usepackage{accents}
\usepackage{fancybox}
\usepackage{slashbox}
\usepackage{bm}
\usepackage{algorithm}

\ifpdf
  \DeclareGraphicsExtensions{.eps,.pdf,.png,.jpg}
\else
  \DeclareGraphicsExtensions{.eps}
\fi

\numberwithin{theorem}{section}
\usepackage{algorithmic}
\usepackage{graphicx}
\usepackage{subfigure}
\usepackage{color}
\jno{ }

\begin{document}

\title{Robust Solutions of Nonlinear Least Squares Problems via Min-max Optimization}
\shorttitle{Robust Nonlinear Least Squares Problems}
\author{%
{\sc Xiaojun Chen\thanks{Corresponding author. Email: xiaojun.chen@polyu.edu.hk}
} \\[2pt]
Department of Applied Mathematics, The Hong Kong Polytechnic University,
 Kowloon, Hong Kong, China\\
 {\sc and }\\
 {\sc C.T. Kelley} \thanks{Email: Tim\_Kelley@ncsu.edu\\ {\bf Revised 20 December 2024}
}\\[2pt]
 North Carolina State University,
Department of Mathematics,
Box 8205, Raleigh, NC 27695-8205, USA
 }

\shortauthorlist{X. Chen and C.T. Kelley}

\maketitle

\begin{abstract}
{This paper considers  robust solutions to a class of nonlinear least squares problems using min-max optimization approach.
 We give an explicit formula for the value function
of the inner maximization problem and show the existence of global minimax points.
We establish error bounds from any solution of the nonlinear least squares  problem to the solution set of the robust nonlinear least squares problem.
Moreover, we propose a smoothing method for finding a global minimax point of the min-max problem by using the formula and show that finding an $\epsilon$ minimax critical point of the min-max problem needs at most $O(\epsilon^{-2} +\delta^2 \epsilon^{-3})$
evaluations of the function value and gradients of the objective function, where
$\delta$ is the tolerance of the noise. Numerical results of integral equations with uncertain data demonstrate the robustness of solutions of our approach and unstable behaviour of
least squares solutions disregarding uncertainties in the data.}
% Keywords:
{Min-max optimization; nonlinear least squares problems; numerical solution of integral equations; nonsmooth nonconvex optimization, complexity }
\end{abstract}

\section{Introduction}

Let $\F:\R^n \to \R^m $ be a continuous function and $\|\cdot\|$ denote the Euclidean norm. The nonlinear least squares problem
\vspace{-0.02in}
\begin{equation}\label{nls}
\min_\x \|\F(\x)\|^2
\end{equation}
is widely used to model problems in science and engineering. When the problems contain uncertain or noisy data, it is important to find a robust solution. In this paper, we consider the following min-max problem
\begin{equation}\label{eq2}
\min_\x\max_{\y\in \Omega} f(\x,\y):= \|\F(\x)-\C\y\|^2,
\end{equation}
for the robust nonlinear least squares problem
where $\C \in \mathbb{R}^{m\times r}$ with $r \le m$,
$\Omega=\{\y \in \mathbb{R}^r\, : \, \|\y\|_\infty \le \delta\}$, and
$\delta\ge 0$.
If $\delta=0$, then problem (\ref{eq2}) reduces to problem (\ref{nls}).  In some applications, the perturbation vector has a known structure and can be present as
a finite linear combination of columns of $C$. See \citet{robustls} and \S 5.2 in this paper.

A particular example is from the
first kind Fredholm  nonlinear integral equation with noisy
data, which can be written in the following form:
\begin{equation}\label{VF}
h(x)=\int^1_0 p(x,y)q(x,u(y)) \ dy,
\end{equation}
where $h, p, q$ are continuous functions.
The low dimensional projection method is a popular numerical method
to find an approximate
solution of (\ref{VF}).  For example, using
\[
\tilde{u}(\vx)=\sum^n_{i=1} \x_{i} T_{i-1}(x),
\]
where $T_i$ is the $i$th degree Chebyshev polynomial on $[0,1]$, we obtain a nonlinear least squares problem  (\ref{nls}). If $h$ contains uncertain or noisy data, we can use
the min-max model (\ref{eq2}) to
find a robust least squares solution.

 Nonlinear least square problems are also often seen in data sciences and machine learning, for example, training convolutional neural network (CNN). Given input and output data sets $\{\a_i\in \mathbb{R}^{\hat{n}}\}_{i=1}^N$ and $\{\b_i\in \mathbb{R}^k\}_{i=1}^N$, the training process of a two-layer CNN \citep{Cui,Liu-Chen1,Liu-Chen}
is represented as
\begin{equation}\label{CNN}
\min_\x \frac{1}{N}\sum_{i=1}^N \|\b_i-\U\sigma(\W\sigma(\V\a_i+\u) +\v)\|^2,
\end{equation}
where  $\U\in \mathbb{R}^{k\times l}$, $\mw\in \mathbb{R}^{l\times l}$,
$\mv\in \mathbb{R}^{l\times \hat{n}}$ are the weight matrices,
$\u\in  \R^l, \v\in \R^k$ are the bias vectors, $\sigma$ is an
activation function, e.g., the ReLU activation function
$\sigma(\vz)=\max(0,\vz)$.  Let
$$
\x=({\rm vec}(\U)^\top,{\rm vec}(\V)^\top, {\rm vec}(\W)^\top, \u^\top,\v^\top)^\top \in \mathbb{R}^{n}
$$
    with $n={kl+l^2+l\hat{n}+l+k}$ and $m=Nk$.  When we study a big set of data, $m$ can be much bigger than $n$.  Problem (\ref{CNN}) can be written as problem (\ref{nls}), which is an overdetermined least squares problem.
When the output data $\{\b_i\in \mathbb{R}^k\}_{i=1}^N$ contain noise,  the min-max problem (\ref{eq2}) can be used to find a robust least squares solution.

First order methods including decent-ascent methods have been widely used to find the first order stationary points of min-max problems \citep{D-Jordan,Jin,Jiang-Chen,YangJ}. However, the first order stationary points defined by gradients of the objective function  $f$ can lead to incorrect solutions of min-max problems. In subsection 1.2,  we show with an example of robust linear least square problems that solving the first order optimality
conditions can lead to incorrect results.

This paper makes three contributions.
\begin{itemize}
\item We use a change of variable in $\vy$
to give an explicit formula for the value function of the inner  maximization problem
\begeq
\label{inner}
\varphi(\x)=
\|\F(\x)\|^2+2\delta \| \mc^\top \mf (\x)\|_1+
\| \mc \|_F^2 \delta^2=\max_{\y\in \Omega} f(\x,\y),
\endeq
for any $\x\in \mathbb{R}^n$, where $ \| \mc \|_F$
%\[
%\| \mc \|_F = \left( \sum_{i=1, j=1}^{m, r} C_{ij}^2 \right)^{1/2}
%\]
is the Frobenius norm of $\mc$.

\item
We propose a smoothing method for finding a global minimax point of the min-max problem by using the function $\varphi$ and show that finding an $\epsilon$ minimax critical point of the min-max problem (\ref{eq2}) needs at most $O(\epsilon^{-2} +\delta^2 \epsilon^{-3})$
evaluations of the function value and gradients of the objective function.

 %We apply the formula for $\varphi$
%to show that the method proposed by Gratton, Simon and Toint \cite{NLS1}
%can be used to find a  first order $\epsilon$-approximate necessary minimax point of  problem (\ref{eq2}) with optimal complexity.

\item We give error bounds from any solution of nonlinear least squares problem (\ref{nls}) to the set of minimizers of $\varphi.$ We give an explicit form
    for the difference of
    worst-case residual errors of a least squares solution and a robust least squares solution, and show  the difference increases linearly as the tolerance of the noise increases.
\end{itemize}

\subsection{Important concepts}

In this subsection we review some subtle concepts for min-max problems.
There are several forms of stationarity and they differ in important
ways. These ideas will reappear throughout the paper.

A point $(\x^*, \y^*)\in \mathbb{R}^n\times \Omega$ is called a
{\bf saddle point} \citep{von-Neumann-1928} of problem (\ref{eq2}) if  for
any $\x \in \mathbb{R}^{m}, \y\in \Omega,$ we have
\begin{equation}\label{saddle}
 f(\x^*, \y)\le f(\x^*, \y^*)\le f(\x, \y^*).
\end{equation}

A point $(\x^*, \y^*)\in \mathbb{R}^n\times \Omega$ is called a {\bf local saddle point}  of problem (\ref{eq2}) if (\ref{saddle}) holds in a neighborhood of
$(\x^*, \y^*)$.

A point $(\x^*, \y^*)\in \mathbb{R}^n\times \Omega$ is called a {\bf global minimax point} \citep{Jin} of problem (\ref{eq2}) if

\[
\max_{\y\in \Omega} f(\x^*, \y)\le f(\x^*, \y^*)
\le \min_\x \max_{y\in \Omega} f(\x, \y).
\]

A point $(\x^*,\y^*)\in \mathbb{R}^n\times \Omega$ is called a
{\bf local minimax point} of problem (\ref{eq2}), if there exist a
$\rho_0>0$ and a function $\tau:\mathbb{R}_+\rightarrow \mathbb{R}_+$
satisfying
$\tau(\rho)\rightarrow 0$ as $\rho\rightarrow 0$,
such that for any $\rho\in (0,\rho_0]$ and any
$(\x,\y)\in\mathbb{R}^n\times \Omega$ satisfying $\|\x-\x^*\|\leq \rho$ and $\|\y-\y^*\|\leq \rho$, we have

\[
f(\x^*,\y) \leq f(\x^*,\y^*) \leq \max_{\y'\in\{\y\in \Omega: \|\y-\y^*\|\leq \tau(\rho)\}}f(\x,\y').
\]

By Theorem 3.11 in \citet{Jiang-Chen}, if $f$
is directionally differentiable then
a local minimax point $(\x^*,\y^*)\in \mathbb{R}^n\times \Omega$
of problem (\ref{eq2}) satisfies the following first order condition
\begin{equation}\label{f-first}
\begin{array}{ll}
d_\x f(\x^*,\y^*;\v)=0, \quad \forall \, \v\in \mathbb{R}^n\\
d_\y f(\x^*,\y^*;\u)\le 0, \quad \forall \,  \u \in \mathrm{T}_\Omega(\y^*),
\end{array}
\end{equation}
where $d_\x f(\x^*,\y^*;\v)$ and $d_\y f(\x^*,\y^*;\u)$
denote the directional derivatives of $f$ at $(\x^*, \y^*)$
along the directions $\v$ and $\u$, respectively,
and
\[
\mathrm{T}_\Omega(\y)=:
\left\{ \w: \exists~ \y^k\overset{\Omega}{\to} \y,
t^k\downarrow 0 ~ \text{such that}~ \lim_{ k\to \infty }
\frac{\y^k-\y}{t^k} = \w \right\}
\]
is the tangent cone to $\Omega$ at $\y\in \Omega$ \citep[Definition 6.1]{RW2009variational}.   When $\F$ is differentiable, the first order optimality condition of problem (\ref{eq2}) has the following form
\begin{eqnarray}\label{first}
\begin{array}{ll}
\F'(\x)^\top(\F(\x)-\C\y)=0\\
\y={\rm Proj}_\Omega(\y-\alpha\C^\top(\F(\x)-\C\y)),
\end{array}
\end{eqnarray}
where $\alpha$ is any positive scalar.

Since the objective function is continuous and the feasible set $\Omega$ is bounded in problem (\ref{eq2}), the solution set

\[
\Omega^*(\x)=\mathop{\mathrm{argmax}}_{\y\in \Omega}f(\x,\y)
\]
is nonempty and bounded for any $\x\in \mathbb{R}^n$. Hence, the function $\varphi$ is well-defined.
However, it is known that solving a convex quadratic maximization over a box feasible set is NP-hard \citep{NP-hard}.
People often consider the first order optimality condition for nonconvex-nonconcave min-max problem \citep{D-Jordan,Jiang-Chen,YangJ}.
 However a solution of this first order optimality condition can be meaningless for (\ref{eq2}), unless the problem has a local minimax problem, for example, convex-concave saddle point problems with cardinality penalties \citet{BXChen2024}.
\subsection{Example}
Consider problem (\ref{eq2}) with $\F(\x)=\A\x-\b$ and $\C=\I$. Suppose that $\A\in\R^{n\times n}$  is nonsingular. Then any pair $(\x, \y)=(\A^{-1}(\b+\y), \y)$ with $\y\in \Omega$  satisfies
the first order condition (\ref{first}). On the other hand,
the optimal solution of the inner maximization problem has the closed form

\begin{equation}\label{y-form}
\y_i(\x)=\mathop{\mathrm{argmax}}_{|\y_i|\le \delta}(\A\x-\b -\y)_i^2= \delta \left\{ \begin{array}{ll}
1 &   {\rm if} \, (\A\x-\b)_i <0\\
\{-1, 1\} & {\rm if} \,(\A\x-\b)_i=0\\
-1 &  {\rm if }\, (\A\x-\b)_i>0,
\end{array}
\right.\quad i=1,\ldots,n.
\end{equation}
Hence
$$\min_\x\max_{\y\in \Omega}  \|\A\x-\b-\y\|^2=\min_\x \|\A\x-\b-\y(\x)\|^2=n\delta^2$$
at $(\x^*, \hat{\y})=(\A^{-1}\b, \delta\c)$,
where $\c$ is a vector with
all elements being 1 or -1. In fact, by the
Danskin theorem, the inner function $\|\A\x-\b-\y(\x)\|^2$ is strongly convex.
Hence, the minimization problem has only one solution $\x^*=\A^{-1}\b.$
It is worth noting that  $(\x^*, \hat{\y})$ is not a saddle point
 of $f(\x,\y)=\|\A\x-\b-\y\|^2,$ for any $\hat{\y}$ with $\hat{\y}_i\in \{\delta, -\delta\}$, since
$$f(\x^*,\hat{\y})= n\delta^2 > f(\A^{-1}(\b+\hat{\y}),\hat{\y})=0.$$

In fact, $f$ does not have a saddle point, since
$$n\delta^2 = \min_\x\max_{\y\in \Omega}  f(\x,\y)\neq \max_{\y\in \Omega} \min_\x f(\x,\y)=0.$$
\noindent See \citet[Theorem 1.4.1]{Pang-book2}. Moreover, due to the arbitrary choice of $\delta $ for the closed form solution
(\ref{y-form}) of the maximization problem, we can easily find that $(\x^*,\hat{\y})$ is not a local saddle point.

On the other hand, we  find

$$
\max_{\y\in \Omega} f(\x^*, \y)=f(\x^*, \hat{\y})=\min_\x \max_{y\in \Omega} f(\x, \y).
$$

\noindent Hence $(\x^*, \hat{\y})$ is a global minimax point of $f(\x,\y)=\|\A\x-\b-\y\|^2$.
However,  $(\x^*, \hat{\y})$ is not a local minimax point of  $f$. Indeed, consider $\hat{\y}=\delta \e$, where $\e$ is the vector with
all elements being 1.

Let $\rho < \frac{\delta}{2\|\ma\|}$ and $\tau(\rho)<\frac{\delta}{2}$ satisfying $\tau(\rho) \to 0$ as $\rho \to 0$.
For any $(\x,\y)\in \mathbb{R}^n\times \Omega $ satisfying $\|\x-\x^*\|\le \rho$
and $\|\y-\delta \e\|\le \rho$, we have

$$|\ma \x-\b|_i\le \|\ma \x-\b\|=\|\ma (\x-\x^*)\|\le \|\ma\|\rho\le \delta-\frac{\tau(\rho)}{2}$$
 and for $i=1\ldots, n$,
$$
\y_i(\x)=\mathop{\mathrm{argmax}}_{|\y_i|\le \delta, |\y_i-\delta|\le \tau(\rho)}(\ma\x-\b -\y)_i^2= \delta \left\{ \begin{array}{ll}
1 &   {\rm if} \, (\A\x-\b)_i <0\\
1 & {\rm if} \,(\A\x-\b)_i=0\\
1 &  {\rm if }\, 0<(\A\x-\b)_i\le \delta-\frac{\tau(\rho)}{2}.
\end{array}
\right.
$$
Let $\x^1=\ma^{-1}(\b+\epsilon \e)$ with $\epsilon \le \rho/\|\ma^{-1}\e\|$. Then
we obtain $\|\x^1-\x^*\|\le \rho$ and
\begin{eqnarray*}
n\delta^2=f(\x^*,\hat{\y})&>&\max_{\|\y\|_\infty\le \delta, \|\y-\delta \e\|_\infty\le \tau(\rho)}\|\A\x^1-\b -\y\|^2
= \|\epsilon\e -\y(\x)\|^2\\
&=&(\epsilon-\delta)^2n\ge \min_\x\max_{\|\y\|_\infty\le \delta, \|\y-\delta\e\|_\infty\le \tau(\rho) } f(\x,\y).
\end{eqnarray*}

Hence, $(\x^*, \hat{\y})$ is not a local minimax point of  $f$.
Moreover, it is easy to see that $(\x^*, \hat{\y})$ does not satisfy the first order condition (\ref{first}), since $\ma^\top(\ma\x^*-\vb-\hat{\y})=-\delta\ma^\top \e\neq 0$. On the other hand, if $(\x', \y')$ satisfies the first order condition (\ref{first}), then from the nonsingularity of $\ma$, we have $f(\x',\y')=0$, which implies
$(\x', \y')$ cannot be a local minimax point or a global
minimax point, since $0=f(\x',\y')\not\ge\max_{\y\in \Omega}f(\x', \y)>0$.

Note that for nonconvex-nonconcave min-max optimization, a local minimax point must satisfy the first order condition, but  a global minimax point is not necessarily  a stationary point \citep{Jiang-Chen}.

In this example we present a min-max problem for robust linear
least squares problems which does not have a saddle point, a local saddle point and a local minimax point, while it has first order stationary points and global minimax points. However, the set of first order stationary points and
the set of global minimax points do not have a common point.

This example also shows that the study of the value function of the inner maximization problem is important.  In \S~\ref{subsec:value},  we give a proof of closed form expression (\ref{inner}) for the inner
value function. Using the inner value function, we prove the existence of a minimax point of the robust nonlinear least squares problem (\ref{eq2}). In \S~\ref{sec:stable}, we
apply the formula to expose some connections between
solutions of (\ref{eq2}) and those of (\ref{nls}), and give error bounds from any solution of nonlinear least squares problem (\ref{nls}) to the set of minimizers of $\varphi.$
In \S~\ref{smooth},
we propose a smoothing method with the formula to obtain a complexity bound
for finding an $\epsilon$ minimax critical point of
a global  minimax point of (\ref{eq2}). In \S~\ref{LsvRLs}, we compare least squares solutions and
robust least squares solutions in the worst-case, and show the difference of their worst-case residual errors
increases linearly as the tolerance of the noise increases. Moreover, we present numerical results of integral
equations with uncertain data to show the robustness of solutions of our approach.

\section{The value function}
\label{subsec:value}

The study of the inner function of the maximization problem is
very important for solving min-max problem (\ref{eq2}).
Denote the value function by
$$\varphi(\x)=\max_{\y\in \Omega} f(\x,\y).$$

We consider a change of variables. Let
\[
\mc = \mU_\mc \Sigma_\mc \mv^\top_\mc
\]
be the singular value decomposition (SVD) of $\mc$. Define
\[
{\hat \vy} = \mv^\top_\mc \vy
\mbox{ and }
{\hat \mc} = \mU_\mc \Sigma_\mc.
\]
Then the QR factorization of ${\hat \mc}$ is $\mU_\mc \Sigma_\mc$ and the
R factor is therefore diagonal. We argue that replacing $\vy$ and $\mc$
with ${\hat \vy}$ and ${\hat \mc}$ is consistent with the modeling assumption
that the perturbation $\mc \vy$ be small and that rotating the coordinates
for $\vy$ does no harm. One can compare the new constraint
$\| {\hat \vy} \|_\infty \le \delta$ with the constraint on $\vy$
to see this by
\begeq
\label{eq:compare}
\| \vy \|_\infty = \| \mv_\mc {\hat \vy} \|_\infty
\le \| \mv_\mc \|_\infty \| {\hat \vy} \|_\infty \le \sqrt{r} \delta.
\endeq

So we will now assume that the change of variable has been made.

\begin{assumption}
\label{ass:qr}
$\mc$ has full column rank.
Let $\mc = \mq \mr$ be the QR factorization of $\mc$, where $\mr$ is
($r \times r$) diagonal with positive diagonal entries.
\end{assumption}

\begin{theorem}\label{thm2.1}
Let Assumption~\ref{ass:qr} hold.
Then the inner function $\varphi$ is given by (\ref{inner}).
\end{theorem}
\begin{proof}
We note that the orthogonality of $\mq$ implies that
\[
\| \mr \|_F = \left( \sum_{j=1}^r \mr_{jj}^2 \right)^{1/2}
= \| \mc \|_F = \left( \sum_{i=1, j=1}^{m, r} \mc^2_{ij} \right)^{1/2}.
\]
Hence (\ref{inner}) is equivalent to
\begin{equation}\label{inner0}
\varphi(\x)=\|\F(\x)\|^2+2\delta \| \mc^\top \mf (\x)\|_1+
\| \mr \|_F^2 \delta^2.
\end{equation}

We will prove (\ref{inner0}).
Our task is to compute the inner function
\[
\varphi(\x) =  \max_{\vy \in \Omega} \| \mf(\vx) - \mc \vy \|^2.
\]

The orthogonal projection onto the range of $\mc$ is
\[
\mP = { \mq} { \mq}^\top
\]
which is an $m \times m$ matrix. Since $\mP$ is an orthogonal
projector and $\mP \mc = \mc$, we have
\begeq
\label{eq:decomp}
\| \mf(\vx) - \mc \vy \|^2 = \| (\mi - \mP) \mf(\vx) \|^2
+ \| \mP \mf(\vx) - \mc \vy \|^2
\endeq
by the Pythagorean theorem. The term $\| (\mi - \mP) \mf(\vx) \|^2$
does not depend on $\vy$ so we have reduced the problem to
computing
\begeq
\label{eq:reduce}
\varphi_0(\vx) = \max_{\vy \in \Omega} \| \mP \mf(\vx) - \mc \vy \|^2.
\endeq

Using the QR factorization again and the orthogonality
of ${ \mq}$ we get
\[
\begin{array}{ll}
\| \mP \mf(\vx) - \mc \vy \|^2 & = \| {  \mq} { \mq}^\top
\mf(\vx) - {  \mq} { \mr} \vy \|^2 \\
& = \| { \mq}^\top \mf(\vx) - { \mr} \vy \|^2.
\end{array}
\]

Setting $\xi = { \mq}^\top \mf(\vx)$ gives us the diagonal form
of the maximization problem
\begeq
\label{eq:final}
\varphi_0(\vx) = \max_{\vy \in \Omega} \| \xi - {\mr} \vy \|^2.
\endeq

We can solve this problem explicitly. Since $\mr$ is diagonal with
positive diagonal elements $\mr_{ii}$, we obtain
\[
\y_i =\delta  \left\{
\begin{array}{ll}
-1 & \mbox{if $\xi_i > 0$}\\
\{ -1, 1 \} & \mbox{if $\xi_i =0$}\\
1 & \mbox{if $\xi_i < 0$,}
\end{array}
\right. \quad i=1,\ldots,r.
\]
Hence, from $\xi = \mq^\top \mf(\vx)$, we have
\begeq
\label{eq:mmsol0}
\varphi_0(\vx) =
\sum_{j=1}^r (|[ \mq^\top \mf(\vx) ]_j| + \mr_{jj} \delta)^2.
\endeq
Here $[ \cdot  ]_i$ denotes the $i$th component of a vector if the
vector is a complicated expression, such as $\mq^\top \mf(\vx)$.

Since $\mP = \mq \mq^\top$ we have
\[
\sum_{j=1}^r |[ \mq^\top \mf(\vx) ]_j|^2
= \| \mq^\top \mf(\vx) \|^2 = \| \mP \mf(\vx) \|^2
\]
and so, using the orthogonality of $\mP$,
\begeq
\label{eq:mmsol}
\begin{array}{ll}
\varphi(\vx) & = \| (\mi - \mP) \mf(\vx) \|^2
+ \sum_{j=1}^r (|[ \mq^\top \mf(\vx) ]_j| + \mr_{jj} \delta)^2 \\
\\
& = \| (\mi - \mP) \mf(\vx) \|^2 + \| \mP \mf(\vx) \|^2 +
2 \delta \sum_{j=1}^r \mr_{jj}| [\mq^\top \mf(\vx) ]_j |+
( \sum_{j=1}^r \mr_{jj}^2 ) \delta^2 \\
\\
& = \| \mf(\vx) \|^2 + 2 \delta \| \mc^\top \mf(\vx) \|_1 +
\| \mr \|_F^2 \delta^2. \\
\end{array}
\endeq
The last equality is a consequence of the diagonality of $\mr$ and the fact
that $\mr \mq^\top = \mc^\top$.
\end{proof}

It is worth noting that  the inner function $\varphi$ is nonsmooth due to
the term $\| \mc^\top \mf(\vx) \|_1$, even if $\mf$ is differentiable.

{\bf Remark 2.1} From the expression (\ref{inner}), the function $\varphi$ has the following
properties.
\begin{description}
\item[(i)]  If $\F$ is continuously differentiable, then $\varphi$ is locally Lipschitz continuous, directionally differentiable, and its directional derivatives satisfy
$$\varphi(\x;\d)=\max_{\y\in \Omega^*(\x)} \d^\top \mf'(\x)^\top(\F(\x)-\y),$$
where $\mf'$ is the Jacobian of $\mf$ and $\Omega^*(\x)=$arg$\min_{\y\in \Omega}\| \mf(\vx) - \mc \vy \|^2.$  If $\Omega^*(\x)=\{\y^*(\x)\}$ is a singleton, then $\varphi$ is differentiable at $\x$ and
$\nabla \varphi(\x)= \mf'(\x)^\top(\mf(\x)-\y^*(\x)).$

\item[(ii)] If the set of minimizers of $\|\F(\x)\|^2$ is nonempty and bounded, then the set of minimizers of $\varphi$ is nonempty and bounded, which can be seen from the relation of the two level sets
$${\cal L}_\varphi:=\{\x \in \mathbb{R}^n :  \, \varphi(\x)\le \beta \} \subseteq {\cal L}_\F:=\{\x \in \mathbb{R}^n :  \, \|\F(\x)\|^2\le \beta \},$$
where $\beta>0$ is a scalar.
\item[(iii)] If $\|\mc^\top \F(\x)\|_1$ has a minimizer $\x^*$ and there is $\gamma>0$ such that
for any $\x\not\in \mx_\gamma:=\{ \x \in \mathbb{R}^n : \, \|\x-\x^*\|\le \gamma\}$,
$\|\F(\x)\|^2> \|\F(\x^*)\|^2$, then $\varphi$ has a minimizer in $\mx_\gamma$, which can be seen from the continuity of $\varphi$ and boundedness of $\mx_\gamma$ with
\begin{eqnarray*}
\varphi(\x)-\varphi(x^*)&=&\|\F(\x)\|^2
-\|\F(\x^*)\|^2+2\delta(\|\mc\F(\x)\|_1-\|\mc\F(\x^*)\|_1)\\
&\ge& \|\F(\x)\|^2-\| \F(\x^*)\|^2>0, \quad \x\not\in \mx_\gamma
\end{eqnarray*}
and
$$\min_{\x\in \mathbb{R}^n} \varphi(\x)\le \min_{\x\in \mx_\gamma} \varphi(\x)
\le \varphi(\x^*).$$
 The function $\|\mc^\top \F(\x)\|_1$ can have infinitely many minimizers, but only the minimizers at which $\|\F(\x)\|$ has small values
are interesting for the robust least squares problems.
For example, $\F(\x)=(\x, 1-\x)^\top$ for $\vx\in \mathbb{R}$ and $\mc\in \mathbb{R}^{2\times 2} $ is the identity matrix. Any point $\x^*\in [0,1]$ is a minimizer of $\|\F(x)\|_1$, but only at $\x^*=\frac{1}{2}$, $\|\F(\x^*)\|<\|\F(\x)\|$
for $\x\not\in \mx_\gamma=\{\x^*\}$.
\item[(iv)] If $\F$ is linear, then $\varphi$ is convex and $\x^*$ is a minimizer of $\varphi$ if and only if
\begin{equation}\label{xl1}
0\in \nabla \|\F(\x^*)\|^2 +2 \delta \partial \|\mc^\top F(\x^*)\|_1.
\end{equation}
If $0\in {\rm int}\partial \|\mc^\top \F(\x^*)\|_1$, then  there is $\bar{\delta}>0$ such that for any
$\delta\ge\bar{\delta}$, (\ref{xl1}) holds, and thus $\x^*$ is a solution of $\varphi$. (See Example 5.1).  Here int$\partial \|\mc^\top \F(\x^*)\|_1$ means that the interior of the Clarke subdifferential of $\|\mc^\top \F(\cdot)\|_1$ at $\x^*$.
 \end{description}

\begin{theorem}\label{minimaxp}
Let Assumption~\ref{ass:qr} hold. If $\varphi$ has a minimizer, then
problem (\ref{eq2}) has a global minimax point and $(\x^*, \y^*)$ is a global minimax point of  (\ref{eq2}) for any
$$\x^* \in \mathop{\mathrm{argmin}}_{\vx} \varphi(\x) \quad {\rm  and} \quad \y^*\in \mathop{\mathrm{argmax}}_{\y\in \Omega} f(\x^*, \y).$$
\end{theorem}
\begin{proof}
Let $\x^*$ be a minimizer of $\varphi(\x)$. Since $f$ is continuous and $\Omega$ is bounded,  there is $\y^*$ such that
\noindent
\begin{equation}\label{th22}
f(\x^*,\y^*)=\max_{\y\in \Omega}f(\x^*,\y).
\end{equation}

Now we show $(\x^*,\y^*)$ is a global minimax point of  (\ref{eq2}). From (\ref{th22}),
we have  $\max_{\y\in \Omega}f(\x^*,\y)\le f(\x^*,\y^*).$
On the other hand, we have
$$\min_\vx\max_{\vy\in \Omega} f(\vx, \vy)=\min_\vx\varphi(\vx)=\varphi(\vx^*)
=\max_{\vy\in \Omega} f(\vx^*, \vy)=f(\vx^*,\vy^*).$$

Hence $(\x^*, \y^*)$ is a global minimax point of  (\ref{eq2}).
\end{proof}

%From Theorem \ref{minimaxp}, we can develop algorithms for
%finding a minimax point
%$(\x^*, \y^*)$ of  (\ref{eq2}) using the following two steps:
%\begin{description}
% \item[Step 1] Find $\x^* \in {\rm arg}\min_\vx \varphi(\x).$
% \item[Step 2] Find $\y^*\in {\rm arg}\max_{\y\in \Omega} f(\x^*, \y).$
%\end{description}

\section{Solution sets of problems (\ref{nls}) and (\ref{eq2}) }
\label{sec:stable}

In this section, we study some interesting relations between the nonlinear least squares problem (\ref{nls}) and the robust nonlinear least squares problem (\ref{eq2}) by using  (\ref{inner}). Denote the solution set of (\ref{nls}) and the set of minimizers of $\varphi$ defined in  (\ref{inner}), respectively, by
\[
\mx^* :=\mathop{\mathrm{argmin}}_\vx \,    \| \mf(\vx) \|^2
\quad {\rm and} \quad
\hat{\mx}:=\mathop{\mathrm{argmin}}_\vx  \,   \varphi(\vx).
\]

We call $\mx^*$ and $\hat{\mx}$ are the nonlinear least squares solution set and
the robust nonlinear least squares solution set, respectively. We call $\x^*\in \mx^*$ and $\hat{\x}\in \hat{\mx}$ are a nonlinear least squares solution  and
a robust nonlinear least squares solution, respectively.

\subsection{Zero residual}

Note that $\varphi({\hat \vx}) \le \varphi(\vx^*)$ and
$\| \mf(\vx^*) \|^2 \le \| \mf({\hat \vx}) \|^2$ for $\hat{\x}\in \hat{\mx}$ and
$\x^*\in \mx^*$. Hence,
\begeq
\label{eq:l1opt}
\begin{array}{ll}
\varphi({\hat \vx})& =
\| \mf({\hat \vx}) \|^2 + 2 \delta \| \mc^\top \mf({\hat \vx}) \|_1 +\| \mc \|_F^2 \delta^2\\
\\
& \le
\varphi({\vx^*}) =
\| \mf({\vx^*}) \|^2 + 2 \delta \| \mc^\top \mf({\vx^*}) \|_1 +\| \mc \|_F^2 \delta^2
\\
\\
& \le \| \mf({\hat \vx}) \|^2 + 2 \delta \| \mc^\top \mf({\vx^*}) \|_1+\| \mc \|_F^2 \delta^2.
\end{array}
\endeq
Therefore
\begeq
\label{eq:1vs2}
\| \mc^\top \mf({\hat \vx}) \|_1 \le \| \mc^\top \mf({\vx^*}) \|_1.
\endeq
When $\mc=\mi$, we can see that a least squares solution of (\ref{nls}) is not
necessarily a least 1-norm solution of $\min_\vx \|\mf(\vx)\|_1$.

When the equality holds in (\ref{eq:1vs2}), then we have $\|\mf({\hat \vx}) \|^2 =\| \mf({\vx^*}) \|^2$ and  $\varphi({\hat \vx}) =
\varphi({\vx^*})$. Hence the nonlinear least squares solution set and
the robust nonlinear least squares solution set are the same, that is  $\mx^*=\hat{\mx}$.
A special case is the zero residual problem.

\begin{proposition}
\label{th:nleq}
Let Assumption~\ref{ass:qr} hold. If $\mf(\vx^*) = 0$, then $(\vx^*, \delta\c)$ is a global minimax point of (\ref{eq2}), where $\c$ is a vector with
all elements being 1 or -1.
\end{proposition}
\begin{proof}
If  $\mf(\vx^*) = 0$ then
the nonsmooth term $
2 \delta \| \mc^\top \mf(\vx^*) \|_1 = 0.
$
Hence the two solution sets $\mx^*$ and $\hat{\mx}$ are the same from (\ref{eq:l1opt}). Moreover, from
$$\y^* \in \mathop{\mathrm{argmax}}_{\vy\in \Omega} \|\mf(\vx^*)-\mc^\top\vy  \|^2=\mathop{\mathrm{argmax}}_{\vy\in \Omega} \|\mc^\top\vy  \|^2,$$
\noindent
we have $\vy^*=\delta\c$, where $\c$ is a vector with
all elements being 1 or -1.
\end{proof}

However, as we saw in the example in subsection 1.1, $(\vx^*, \delta \c)$
may not be a saddle point, a local saddle point, a local minimax point or a stationary point of (\ref{eq2}).
Moreover, solutions of the first order
conditions (\ref{f-first}) or (\ref{first}) may not be global minimax points
of (\ref{eq2}).

\subsection{Error bounds: linear least squares problems}

We first consider the linear least squares problem with
\[
\mf(\vx) = \ma \vx - \vb
\]
where $\ma \in \mathbb{R}^{m \times n}$ and $\vb \in\mathbb{R}^m$. Let
\begin{equation}\label{LLS2}
\varphi_{\rm min}:= \min_\vx \varphi(\x):= \|\ma\vx - \vb\|^2 + 2\delta \|\mc^\top (\ma \vx - \vb) \|_1 +\|\mc\|^2_F\delta^2.
\end{equation}
Since the objective function of problem (\ref{LLS2}) is piecewise quadratic convex with nonnegative function values, problem (\ref{LLS2})
has a minimizer \citep{Frank}.

The subgradient of $\varphi$ at $ \x$ is as follows

$$\partial\varphi(\x)=2\ma^\top (\ma \vx- \vb) +  2\delta  \partial \|\mc^\top(\ma\vx - \vb) \|_1.$$

\noindent Let $\mb=\mc^\top\ma$ and $\vc=\mc^\top \vb$. Then

\begin{eqnarray*}
\partial\|\mc^\top(\ma\vx - \vb) \|_1&=&\sum_{\mb_i\vx-\vc_i>0} \mb_i^\top
- \sum_{\mb_i\vx-\vc_i<0}\mb_i^\top
+\sum_{\mb_i\vx-\vc_i=0} [\min (\mb_i^\top, -\mb_i^\top), \max(\mb_i^\top,-\mb_i^\top)]\\
&\subseteq& [-|\mb|^\top \ve, |\mb|^\top\ve],
\end{eqnarray*}
where $\mb_i$ is the $i$th row of $\mb$, $|\mb|=(|\mb_{i,j}|)\in \mathbb{R}^{r\times n} $, and $\ve\in \mathbb{R}^r$ is the vector with all elements being one.

\begin{theorem}
\label{th:lles}
%Let Assumption~\ref{ass:qr} hold.
For any given $\Gamma > 0$,  there is
$\gamma>0$ such that for any
$\vx^* \in \mx^*$ with $\varphi(\x^*)-\varphi_{\rm min}\le \Gamma$,
\begin{equation}\label{boundx}
{\rm dist} (\x^*,\hat{\mx}) \le 2\delta\gamma^2\sqrt{r} \|\mc^\top\ma\|.
\end{equation}
\end{theorem}
\begin{proof}  From Theorem 2.7 in \citet{WLi}, we have
\begin{equation}\label{bound1}
{\rm dist} (\x^*,\hat{\mx}) \le \gamma\sqrt{\varphi(\x^*)-\varphi_{\rm min}},
\end{equation}
if  $\varphi(\x^*)-\varphi_{\rm min}\le \Gamma$.
Let $\bar{\x} \in \hat{\mx}$ such that $\|\x^*-\bar{\x}\|={\rm dist} (\x^*,\hat{\mx}).$
Since $\varphi$ is convex, we have
$$\varphi(\x^*)-\varphi(\bar{\x})
\le g(\x^*)^\top(\x^*-\bar{\x})\le\|g(\x^*)\|\|\x-\bar{\x}\|=\|g(\x^*)\|{\rm dist} (\x^*,\hat{\mx}),$$
where $g(\x^*)\in \partial \varphi(\x^*)$.

Since $\x^*\in \mx^*,$  we have $2\ma^\top (\ma \vx^*- \vb)=0$ and
\begin{eqnarray*}
\partial\varphi(\x^*)&=&2\ma^\top (\ma \vx^*- \vb) +  2\delta  \partial \|\mc^\top(\ma{\vx^*} - \vb) \|_1\\
&=& 2\delta  \partial \|\mc^\top(\ma{\vx^*} - \vb) \|_1\\
&\subseteq& 2\delta [-|\mb|^\top\ve, |\mb|^\top\ve].
\end{eqnarray*}

Hence from (\ref{bound1}), we find
$${\rm dist} (\x^*,\hat{\mx}) \le \gamma^2\|g(\x^*)\|\le 2\delta \gamma^2 \sqrt{r} \|\mc^\top\ma\|.$$
This completes the proof.
\end{proof}

From the first order necessary optimality conditions of (\ref{nls}) and (\ref{eq2}) with $\mf(\x)=\ma\x-\vb$,
we have
$$\ma^\top (\ma\x^*-\vb)=0, \quad  \ma^\top(\ma\hat{\x}-\vb-\mc\y)=0,$$
for $\x^*\in \mx^*$ and $\hat{\x}\in \hat{\mx}$, which implies
$$\ma^\top \ma(\x^*-\hat{\x})=\mc\y.$$

If  $\ma$ has full column rank, then
$$\|\x^*-\hat{\x}\| \le \|(\ma^\top\ma)^{-1}\mc\|\|\y\|\le \sqrt{r}\delta\|(\ma^\top\ma)^{-1}\mc\|.$$
Note that this error bound is the same $O(\delta)$ as (\ref{boundx}). However, Theorem \ref{th:lles} does not assume that $\ma$ has full column  rank.

In the case of overdetermined linear least squares problems, with $m\gg n>r$,
we often have $(\mc^\top(\ma\x^*-\vb))_i\neq 0,$ for $ i=1,\ldots, r$, where $\x^*$ is the
minimum norm solution of the linear least squares problem
$$\min \|\ma\x-\vb\|^2.$$
Let the SVD of $\ma$ be
\[
\ma = \mU_A \Sigma_A \mv_A^\top.
\]
Then the pseudo-inverse of $\ma$ is
\[
\ma^\dagger = \mv_A \Sigma_A^\dagger \mU_A^\top
\]
and the minimum norm solution of the linear least squares problem is
\[
\vx^* = \ma^\dagger \vb=\mv_A \Sigma_A^\dagger \mU_A^\top\vb=(\ma^\top\ma)^\dagger\ma^\top\vb.
\]
Let
\begin{equation}\label{definition-eta}
\eta=\min_{1\le i\le r}|\mc^\top(\ma\x^*-\vb)|_i.
\end{equation}

\begin{proposition}
If $\eta>0$, then for any positive scalar $\delta$ satisfying
\begin{equation}\label{eta}
\delta<\frac{\eta}{r\|\mc^\top(\ma^\top)^\dagger(\mc^\top\ma)^\top\|_1}
\end{equation}
the minimum norm solution of the robust linear least squares problem (\ref{LLS2}) is
\begin{equation}\label{smooth-bound0}
\hat{\x}= (\ma^\top\ma)^\dagger(\ma^\top\vb-\delta \sum_{i=1}^r(\mc^\top\ma)^\top_i{\rm sgn}(\mc^\top(\ma \x^*-b))_i),
\end{equation}
where  $(\mc^\top\ma)_i$ is the $i$th row of $\mc^\top\ma$ and {\text sgn} is the sign function.
\end{proposition}
\begin{proof}
First we show that
$\min_{1\le i\le r}|\mc^\top(\ma\hat{\x}-\vb)|_i\neq 0$,
which implies that  $\varphi$ is differentiable at $\hat{\x}$.
From the SVD of $\ma$ and the form of $\x^*$, we have
\begin{equation}\label{smooth-bound}
 \hat{\x}= \x^* - \delta(\ma^\top\ma)^\dagger\sum_{i=1}^r(\mc^\top\ma)^\top_i{\rm sgn}(\mc^\top(\ma \x^*-b))_i.
 \end{equation}
Hence for $i=1,\ldots, r$, we have
\begin{eqnarray*}
|\mc^\top(\ma\hat{\x}-\vb)-\mc^\top(\ma\x^*-\vb)|_i&=& \delta |\mc^\top\ma(\ma^\top\ma)^\dagger\sum_{i=1}^r(\mc^\top\ma)^\top_i{\rm sgn}(\mc^\top(\ma \x^*-b))_i|_i\\
&=& \delta |\mc^\top(\ma^\top)^\dagger\sum_{i=1}^r(\mc^\top\ma)^\top_i{\rm sgn}(\mc^\top(\ma \x^*-b))_i|_i\\
&\le & \delta \sum_{i=1}^r|\mc^\top(\ma^\top)^\dagger(\mc^\top\ma)^\top_i|_i\\
&\le& \delta r \|\mc^\top(\ma^\top)^\dagger(\mc^\top\ma)^\top\|_1\\
&<&\eta.
\end{eqnarray*}
By the definition of $\eta$ in (\ref{definition-eta}), we find that  for $i=1,\ldots, r$, $|\mc^\top(\ma\hat{\x}-\vb)|_i\neq 0$ and
$${\rm sgn}(\mc^\top(\ma \x^*-b))_i={\rm sgn}(\mc^\top(\ma \hat{\x}-b))_i.$$
Now we show $\hat{\x}$ is a minimum norm minimizer of $\varphi$. Since $\varphi$ is a convex function and differentiable at $\hat{\x}$, the first order optimality condition is necessary and sufficient for the optimality, which holds at $\hat{\x}$ as the follows
\begin{eqnarray*}
\nabla \varphi(\hat{\x})&=&
2\ma^\top(\ma \hat{\x}-b)+2\delta \sum_{i=1}^r(\mc^\top \ma)_i{\rm sgn}(\mc^\top(\ma \hat{\x}-b))_i\\
&=&2 \ma^\top\ma \hat{\x}- 2\ma^\top b+2\delta \sum_{i=1}^r(\mc^\top \ma)_i{\rm sgn}(\mc^\top(\ma \x^*-b))_i\\
&=&0,
\end{eqnarray*}
where we use (\ref{smooth-bound0}). Hence  $\hat{\x}$ is the minimum norm solution of (\ref{LLS2}).
\end{proof}

\vspace{0.02in}
\begin{remark}
If $\F(\x^*)=\ma\x^*-\vb\neq 0$, then we can choose a matrix $\mc\in \mathbb{R}^{m\times r}$ with rank$(\mc)=r$ such that $\eta= \min_{1\le i\le r}|\mc^\top(\ma\x^*-\vb)|_i>0$.
For example, let
$$\mc_{i,j}=\left\{\begin{array}{lr}
1 &  {\rm if} \,\, \F_{i}(\x^*)\ge 0, \,  \, i\neq j\\
-1 & {\rm if} \,\, \F_{i}(\x^*)< 0, \,  \, i\neq j\\
2r & {\rm if} \, \, \F_{i}(\x^*)\ge 0, \, \,  i=j\\
-2r &{\rm if}  \, \, \F_i(\x^*) <0, \,  \, i=j
\end{array}
\right.
$$
for $i=1,\ldots,m$ and $j=1,\ldots, r$.
 Then $\eta\ge \|\F(\x^*)\|_1$.
Moreover, the submatrix of the first $r$ rows of $\mc$ is strictly diagonally dominant. Hence rank$(\mc)=r$.
\vspace{0.02in}
\end{remark}

From (\ref{smooth-bound}), we also obtain an error bound
\begin{equation}\label{boundS}
{\rm dist}(\x^*, \hat{\mx})\le \|\x^*-\hat{\x}\|\le 2\delta r \|(\ma^\top\ma)^\dagger
(\mc^\top\ma)^\top\|_1
\end{equation}
 in a neighborhood of the minimum norm solution $\x^*$  of the linear least squares problem when $\varphi$ is differentiable at  $\x^*$.
 This special case
is different from the one in Equation (24)  of \citet[page 1044]{robustls}.
 Note that in \citet{robustls}, the set
$\Omega=\{\y \, : \, \|\y\|\le \delta\}$ is defined by the Euclidean norm and a linear least squares solution is a robust linear least squares solution.

\subsection{Error bounds: nonlinear least squares problems}

We use the second order growth condition in \citet[Section 4.4.1]{Shapiro}
to present an error bound from a nonlinear least squares solution to the robust nonlinear least squares
solution set.

We say that the second order growth condition holds at $\hat{\mx}$ if there exists a neighborhood ${\cal N}$ of $\hat{\mx}$ and a constant $\alpha>0$ such that
\begin{equation}\label{growth}
\varphi(\x)\ge \varphi_{\rm min} +\alpha [{\rm dist}(\x, \hat{\mx})]^2, \, \, \, \forall \x \, \in {\cal N},
\end{equation}
where  $\varphi_{\rm min}=\min_\x \varphi(\x).$  If $\varphi$ is a piecewise convex quadratic function, then (\ref{growth}) holds \citep{WLi}.

We say that  $\mc^\top\F$ is Lipschitz continuous in  ${\cal N}$ at $\hat{\mx}$ if there is $L>0$ such that
\begin{equation}\label{L}
\|\mc^\top(\F(\x)-\F(\z))\|\le L\|\x-\z\|,   \quad \quad \forall \x \, \in {\cal N},
\z \in \hat{\mx}.
\end{equation}

\begin{theorem}\label{thm3.4}
Suppose the  second order growth condition (\ref{growth}) and the Lipschitz continuity condition (\ref{L}) hold. Then for any $\x^*\in \mx^* \cap {\cal N}$, we have
\begin{equation}\label{boundN}
{\rm dist} (\x^*, \hat{\mx})\le 2\delta \sqrt{r}L \alpha^{-1}.
\end{equation}
\end{theorem}
\begin{proof}
Let $\hat{\x}\in \hat{\mx}$ such that $\|\x^*-\hat{\x}\|=$dist$(\x^*, \hat{\mx}).$
From condition (\ref{L}), we have
\begin{eqnarray*}
\varphi(\x^*)-\varphi(\hat{\x})&=&
\|\F(\x^*)\|+2\delta\|\mc^\top\F(\x^*)\|_1-\|\F(\hat{\x})\|-2\delta\|\mc^\top\F(\hat{\x})\|_1\\
&\le& 2\delta\|\mc^\top\F(\x^*)\|_1-2\delta\|\mc^\top\F(\hat{\x})\|_1\\
&\le& 2\delta \|\mc^\top(\F(\x^*)-\F(\hat{\x}))\|_1\\
&\le& 2\delta \sqrt{r}\|\mc^\top(\F(\x^*)-\F(\hat{\x}))\|\\
&\le& 2\delta \sqrt{r} L\|\x^* -\hat{\x}\|\\
&=& 2\delta \sqrt{r} L {\rm dist}(\x^*, \hat{\mx}),
\end{eqnarray*}
where the first inequality uses $\|\F(\x^*)\|\le \|\F(\hat{\x})\|.$
Together with the second order growth condition (\ref{growth}), we find the error bound (\ref{boundN}).
\end{proof}

In this section, we give upper-bounds of the distance between a least squares solution and the solution set of the robust least squares problem.  Theorem \ref{th:lles} does not reply on a quadratic growth condition and the differentiability.
On the other hand, Theorem \ref{thm3.4}  assumes  a quadratic growth condition and the inequality in (\ref{boundS}) relies on the differentiability of $\varphi$ at the least squares solution. 

\section{Smoothing methods and complexity for minimax problem (\ref{eq2})}
\label{smooth}

The inner maximization function $\varphi$ in (\ref{inner}) is  a nonconvex composite function with a structured nonsmooth term. In this section, we consider smoothing methods and complexity for minimax problem (\ref{eq2}). We consider the following minimization problem
\begin{equation}\label{NLS11}
\min_\vx  \psi(\vx):=\|\mf(\vx)\|^2 + 2 \delta\|\mc^\top \mf(\vx)\|_1,
\end{equation}
and assume that $\F: \mathbb{R}^n\to \mathbb{R}^m$ is a continuously differentiable function in this section.
If $\delta=0$, then problem (\ref{NLS11}) reduces to smooth nonlinear least squares problem (\ref{nls}).

\subsection{Smoothing function}
\label{subsec:smooth}

Using the smoothing convex function $\theta(t,\mu):=\sqrt{t^2+4\mu^2}$ with the smoothing parameter $\mu>0$ of $|t|$ \citep{XChen2012}, we define a smoothing function of  $\psi$   as follows
\[
\Psi(\vx,\mu)=\|\mf(\vx)\|^2 +2\delta\Theta(\mc^\top \mf(\vx),\mu),
\]
where
\begeq
\label{eq:Thetadef}
\Theta(\mc^\top \mf(\vx),\mu)
=\sum^r_{i=1}\theta((\mc^\top \mf(\vx))_i,\mu).
\endeq

For any fixed $\mu>0$, $\Theta(\cdot, \mu)$ is continuously differentiable and
\begin{equation}\label{mu}
0\le \Theta(\mc^\top \mf(\vx),\mu)-\|\mc^\top \mf(\vx)\|_1\le 2r\mu.
\end{equation}

Since $\mf$ is continuously differentiable, $\theta(\cdot, \mu)$ and $|\cdot|$ are convex, we have the gradient consistence \citep{XChen2012}
$${\rm con}\{\lim_{\x\to \bar{\x}, \mu\downarrow 0} \nabla \Theta(\mc^\top \mf(\vx),\mu)\}=\partial \|\mc^\top \mf(\bar{\vx})\|_1,$$
where ``con'' denotes the convex hull and $\partial \|\mc^\top \mf(\vx)\|_1$ is the Clarke gradient of  $\|\mc^\top \mf(\cdot)\|_1$  at $\vx$ \citep{Clarke}. Hence, we have
\begin{equation}\label{first-orderC1}
\begin{array}{ll}
{\rm con}\{{\displaystyle \lim_{\x\to \bar{\x}, \mu\downarrow 0}\nabla \Psi(\vx,\mu)}\}&
={\displaystyle
\lim_{\x\to \bar{\x}} 2\F'(\x)^\top \F(\x)+2\delta {\rm con}\{\lim_{\x\to \bar{\x}, \mu\downarrow 0}\nabla \Theta(\mc^\top \mf(\vx),\mu)}\}\\
&= 2\F'(\bar{\x})^\top \F(\bar{\x})+ 2\delta \partial \|\mc^\top \mf(\bar{\vx})\|_1=\partial \psi(\bar{\vx}).
\end{array}
\end{equation}

It is known that if $\x^*$ is a local minimizer of problem (\ref{NLS11}), then $0\in \partial \psi(\x^*)$ \citep{Clarke}.
A vector $\x^*$ is called a {\bf  Clarke stationary point} of $\psi$ if $0\in \partial \psi(\x^*)$.

From the gradient consistence  (\ref{first-orderC1}), if
$\x^*$ is a local minimizer of problem (\ref{NLS11}),  then there are sequences $\{\x_k\}$ and
$\{\mu_k\}$  such that
$$\lim_{\x_k\to \x^*, \mu_k\downarrow 0}\nabla \Psi(\vx_k,\mu_k)=0.$$

For a fixed $\mu>0$, a vector $\x_\mu$ is called a {\bf stationary point}  of $\Psi(\vx_\mu,\mu)$
if $\nabla \Psi(\vx_\mu,\mu)=0.$ From the gradient consistence, we have
$$\lim_{\mu\downarrow 0}\nabla \Psi(\vx_\mu,\mu)=0$$
and if $\{\x_\mu\}$ is bounded, then there is subsequence $\{\x_{\mu_k}\}$ of $\{\x_\mu\}$ such that
$\{\x_{\mu_k}\}$ converges to a stationary point of $\psi$ as $\mu_k \downarrow 0$.

\subsection{Complexity for  problem (\ref{NLS11})}
\label{subsec:nlsq}

Let the linearization of $\psi$ at $\x$ be defined as follows
$$\phi(\vx,\vs)=\|\mf(\vx)\|^2 + 2(\mf'(\vx)^\top\mf(\vx))^\top \vs + 2 \delta \|\mc^\top (\mf(\vx)+\mf'(\vx)\vs)\|_1.$$
An $\epsilon$ critical point of $\psi$  is defined in \citet{Toint2011} by the linearization of $\psi$ as follows.

We said that $\vx^*$ is an {\bf $\epsilon$ critical point } of $\psi$  if
\begin{equation}\label{order}
\phi(\vx^*,0)-\min_{\|\vs\|\le 1} \phi(\vx^*, \vs)
\le \epsilon.
\end{equation}
When $\epsilon=0$ in (\ref{order}),
 $\vx^*$ is called a {\bf  critical point } of $\psi$, that is,
\begin{equation}\label{first-orderT}
\min_{\|\vs\|\le 1} \phi(\vx^*, \vs)=\phi(\vx^*,0).
\end{equation}
We can verify  that  if $\vx^*$ is a local minimizer of $\psi$, then (\ref{first-orderT}) holds. Indeed, by Taylor's theorem, for any small $\alpha\ge0$, and $\vs\in \mathbb{R}^n$ with  $0<\|\vs\|\le 1$, we have
$$\psi(\vx^*+\alpha \vs)\le\phi(\vx^*,\alpha \vs)+o(\alpha) +o(\alpha)2\delta\|\vs\|=\phi(\vx^*,\alpha \vs)+o(\alpha),$$
and by the convexity of $\phi(\vx^*,\cdot)$, we have
$$\phi(\vx^*,\alpha \vs)=\phi(\vx^*,\alpha \vs +(1-\alpha) 0)\le \alpha\phi(\vx^*,\vs)+(1-\alpha)\phi(\vx^*,0).$$
These two inequalities with $\phi(\vx^*,0)=\psi(\vx^*)$ imply
$$0\le \psi(\x^*+\alpha \vs)-\psi(\x^*)\le  \phi(\vx^*,\alpha \vs)+o(\alpha)-\psi(\x^*)\le
\alpha(\phi(\vx^*,\vs)-\psi(\vx^*)) +o(\alpha).$$
Hence if $\vx^*$ a local minimizer of $\psi$, then  a minimizer $\vs^*$ of (\ref{first-orderT}) must satisfy
$\phi(\vx^*, \vs^*)=\phi(\vx^*,0).$
Moreover, using the convexity of $\phi(\vx^*,\cdot)$ again, we have
\begin{equation}\label{first-orderCT}
0\in \partial \phi(\vx^*,0)= 2\mf'(\vx^*)^\top\mf(\vx^*) + 2\delta \partial \|\mc^\top \mf(\vx^*)\|_1=\partial \psi(\vx^*).
\end{equation}
 Hence $\x^*$ is a critical point of $\psi$ if and only if $\x^*$ is a Clarke stationary point.

 \begin{proposition}\label{complexity}
 Let $\delta >0$ and $\mu\le \epsilon/(8r\delta)$. If $\|\nabla \Psi(\x^*_\mu,\mu)\|$  $\le \epsilon/2$, then
 \begin{equation}\label{pro4-1}
 \phi(\x^*_\mu,0)-\min_{\|\vs\|\le 1} \phi(\vx^*_\mu, \vs)\le \epsilon,
 \end{equation}
 that is, $\x^*_\mu$ is an $\epsilon$ critical point of $\psi$.
 \end{proposition}
 \begin{proof} We verify (\ref{pro4-1})  as follows. Let
 $\vs_\mu\in \mathop{\mathrm{argmin}}_{\|\vs\|\le 1} \phi(\vx^*_\mu, \vs).$ We have
 \begin{eqnarray*}
 &&\phi(\x^*_\mu,0)-\phi(\x^*_\mu, \vs_\mu)\\
 &&=\psi(\x^*_\mu)-\psi(\x^*_\mu)-2(\mf'(\vx^*_\mu)^\top\mf(\vx^*_\mu))^\top \vs_\mu - 2 \delta \|\mc^\top (\mf(\vx^*_\mu)+\mf'(\vx^*_\mu)\vs_\mu)\|_1 + 2\delta \|\mc^\top \mf(\vx^*_\mu)\|_1\\
 &&\le -2(\mf'(\vx^*_\mu)^\top\mf(\vx^*_\mu))^\top \vs_\mu - 2 \delta \Theta(\mc^\top (\mf(\vx^*_\mu)+\mf'(\vx^*_\mu)\vs_\mu),\mu) + 2\delta \Theta(\mc^\top \mf(\vx^*_\mu),\mu) +4\delta r\mu\\
 &&\le -\nabla \Psi(\x^*_\mu,\mu)^\top\vs_\mu + 4r\mu \delta\\
 &&\le \|\nabla \Psi(\x^*_\mu,\mu)\|\|\vs_\mu\| +4r\mu\delta \le \epsilon,
 \end{eqnarray*}
 where the first inequality uses  (\ref{mu}), the second inequality uses the
 convexity and twice continuously differentiability  of $\Theta$ and the last inequality uses $\|\vs_\mu\|\le 1$
 and the assumption of this proposition.
 \end{proof}

Suppose  $\mf'(\cdot)^\top\mf(\cdot)$ and $\mf'(\cdot)$ are Lipschitz continuous with
constants $L_1$ and $L_2$. Then the gradient of  $\Psi$ is Lipschitz continuous
with $L_\Psi:=\gamma(L_1+\delta L_2\mu^{-1})$ where $\gamma$ is a constant which is independent
of $\x$ and $\mu$.  Hence  we can find an $\epsilon$ stationary point of the continuously differentiable function $\Psi(\cdot, \mu)$ for any given $\mu>0$ at most $O(L_\Psi\epsilon^{-2})$ function evaluations using standard steepest descent methods with linesearch or trust-region safeguards \citep{Toint2011,Toint2012,Nesterov}. Hence from Proposition \ref{complexity}, we can choose $\mu=\epsilon/(8r\delta)$ and find $\x^*_\mu$ such that
$\|\Psi(\x^*_\mu, \mu)\|\le \epsilon/2$ at most $O(L_\Psi\epsilon^{-2})$ function and gradient evaluations, which means that we can find
an $\epsilon$ critical
 point of the nonsmooth function $\psi$ at most
 $O(\epsilon^{-2}+\delta^2r\epsilon^{-3})$ evaluations of $\mf$ and $\mf'$
 via a smoothing function. When $r$ is independent of dimension $n, m$, the complexity is $O(\epsilon^{-2}+\delta^2\epsilon^{-3})$, which is independent of the dimension $m,n$.
Moreover, if $\epsilon \ge \delta^2$, then the complexity bound is $O(\epsilon^{-2})$, which is the lower bound for finding stationary points of smooth nonconvex functions \citep{Carmon2020,Toint2011} and independent of the dimension of the problem.
One of the advantages of smoothing methods is to
make use of efficient optimization software for solving
smooth optimization problems.
\vspace{0.05in}

\begin{remark}
In \citet{Toint2011}, Cartis et al propose a trust region method for a composite nonsmooth nonconvex optimization problem
\begin{equation}\label{composite}
\min_\x f(\x)+h(c(\x)),
\end{equation}
where $f:\mathbb{R}^n\to \mathbb{R}$ and  $c:\mathbb{R}^n\to \mathbb{R}^r$ are continuously differentiable
and $h:\mathbb{R}^r\to \mathbb{R}$ is convex. In \citet{NLS1}, Gratton et al generated the results in \citet{Toint2011} to inexact function values, and
  proposed an adaptive regularization algorithm for problem (\ref{composite}). Problem
(\ref{NLS11}) is a special case of (\ref{composite})  and satisfies Assumption
AS.1-AS.4 in \citet{NLS1} when $\F$ is continuously differentiable and the level set
$\{\x \, : \, \psi(\x)\le \psi(\x_0)\}$ is bounded, where $\x_0$ is the initial point of the algorithm. Hence it is interesting to develop efficient codes to implement the algorithms in \citet{Toint2011,NLS1} for finding  an
$\epsilon$ critical point  $\vx$ satisfying (\ref{order}), which needs at most $O(\epsilon^{-2})$
evaluations of $\mf$ and $\mf'$.
\end{remark}

\subsection{Complexity for minimax problem (\ref{eq2})}
For a general min-max problem
\begin{equation}\label{Gminmax}
\min_{\x \in \mathbb{R}^n} \max_{\y\in {\cal Y}} f(\x,\y),
\end{equation}
where ${\cal Y}$ is a convex set, the widely used first order optimality condition is
\begin{equation}\label{Gfirst}
\nabla_\x f(\x,\y)=0, \quad \quad \y={\rm Proj}_{{\cal Y}}(\y+ \alpha \nabla_\y f(\x,\y)),
\end{equation}
where $\alpha$ is a positive scaler. The system of nonlinear equations (\ref{Gfirst})  is a necessary condition for local minimax points of (\ref{Gminmax}) \citep{Jiang-Chen,Jin}.  Many algorithms have been developed to find stationary points satisfying (\ref{Gfirst}) \citep{D-Jordan,Jin,YangJ}. However, as shown in Subsection 1.2,   minimax problem (\ref{eq2}) does not have a local minimax point. Although problem (\ref{eq2}) has stationary points and minimax points, the set of stationary points and the set of minimax points do not have a common point. Naturally, we have to define a meaningful
stationary point when the min-max problem does not have a local minimax point.

\begin{definition}\label{first-order2}
We say $(\hat{\vx}, \hat{\vy})$ is an {\bf   $\epsilon$ minimax critical point}
of the min-max problem (\ref{eq2})
if $(\hat{\x}, \hat{\y})$ satisfies
\begin{eqnarray}\label{order2}
&&\varphi (\x,0)-\min_{\|\vs\|\le 1} \varphi(\vx, \vs)
\le \epsilon,\\
&&\y\in \mathop{\mathrm{argmax}}_{\|\y\|_\infty\le \delta} \|\C^\top(\F(\x)-\C\y)).\label{order3}
\end{eqnarray}

We say $(\hat{\vx}, \hat{\vy})$ is a {\bf minimax critical point}
of the minimax problem (\ref{eq2}) if  $\epsilon=0$ in (\ref{order2}).
\end{definition}

\begin{proposition}\label{first-order3}
If $(\hat{\x}, \hat{\y})$ is a global minimax point of problem (\ref{eq2}), then
(\ref{order2}) holds with $\epsilon=0$, that is $(\hat{\vx}, \hat{\vy})$ is a minimax critical point
of problem (\ref{eq2}).
\end{proposition}
\begin{proof} By the definition of a global minimax point of (\ref{eq2}), we have
$$
\max_{\y\in \Omega} \|\mf(\hat{\vx})-\mc\vy  \|^2\le
\|\mf(\hat{\vx})-\mc\hat{\vy}  \|^2
\le \min_\x \max_{\y\in \Omega} \|\mf(\vx)-\mc\vy  \|^2.
$$
Hence, we find that $\hat{\x}$ is a minimizer of $\varphi$ as follows,
$$\varphi(\hat{\x})=\max_{\y\in \Omega} \|\mf(\hat{\vx})-\mc\vy  \|^2\le
\|\mf(\hat{\vx})-\mc\hat{\vy}  \|^2
\le \min_\x \max_{\y\in \Omega} \|\mf(\vx)-\mc\vy  \|^2=\min_\x \varphi(\x).$$
Following the discussion in Subsection 4.2, $\hat{\x}$ is a critical point of
$\varphi$, and thus (\ref{order2}) holds with $\epsilon=0$.

From $\max_{\y\in \Omega} \|\mf(\hat{\vx})-\mc\vy  \|^2\le
\|\mf(\hat{\vx})-\mc\hat{\vy}  \|^2$, we have  $\hat{\y} \in \mathop{\mathrm{argmax}}_{\vy\in \Omega} \|\mf(\hat{\vx})-\mc^\top\vy  \|^2$.
\end{proof}

From Theorem \ref{thm2.1}, if
Assumption~\ref{ass:qr} holds,
then the inner function $\varphi$ is given by (\ref{inner}). From Proposition 4.1, we can find a $\hat{\x}$ such that (\ref{order2}) holds after at most
$O(\epsilon^{-2}+\delta^2\epsilon^{-3})$ evaluations of $\mf$ and $\mf'.$
Moreover, by Assumption \ref{ass:qr}, $\C^\top \C$ is a diagonal matrix with positive diagonal elements, which implies that we can find $\hat{\y}$ satisfying (\ref{order3}) by
\[
\hat{\y}_i = \delta \left\{
\begin{array}{ll}
-1 & \mbox{if $(\C^\top \mf(\hat{\x}))_i > 0$}\\
\{ -1, 1 \} & \mbox{if $(\C^\top \mf(\hat{\x}))_i=0$}\\
1 & \mbox{if $(\C^\top \mf(\hat{\x}))_i< 0$,}
\end{array}
\right.  \quad i=1,\ldots,r.
\]

Hence, we can have the following theorem for complexity bound of the min-max problem (\ref{eq2}).
\begin{theorem}
Suppose Assumption \ref{ass:qr} holds. We take at most
$O(\epsilon^{-2}+\delta^2\epsilon^{-3})$ evaluations of $\mf$ and $\mf'$ to find an $\epsilon$  minimax critical point of the min-max problem (\ref{eq2}) by a steepest descent method with linesearch or trust-region for minimizing $\Psi(\x, \mu)$   with $\mu=\epsilon/(8r\delta)$.
\end{theorem}

 The complexity of finding a near-stationary
point of the inner maximization function has been studied in some literature for minimax optimization problems.   For example, the complexity results on a two-time-scale gradient descent ascent algorithm (GDA) for solving nonconvex concave minimax problems is provided  in \citep{TLin}, and the complexity results on alternating GDA and smoothed GDA is proved under the Polyak-{\L}ojasiewicz condition in \citep{JYang}.

\section{Least squares solutions vs robust least squares solutions in the worst-case residual errors}\label{LsvRLs}

We call $\x_{ls}$ and $\x^\lambda_{rls}$ a least squares solution
and a $\lambda$-robust least squares solution, respectively if
\[
\x_{ls} \in \mathop{\mathrm{argmin}}\|\mf(\x)\|^2
\]
and
\[
\x^\lambda_{rls} \in \mathop{\mathrm{argmin}}
\|\mf(\x)\|^2+2\lambda \|\mc^\top \mf(\x)\|_1.
\]

In this section, we compare least squares solutions and
$\lambda$-robust least squares solutions in the worse-case residual
error, that is,
\[
Er_\delta(\x_{ls})= \max_{\|\y\|\le \delta} \|\mf(\x_{ls})-\mc^\top\y\|^2
\]
and
\[
Er_\delta(\x^\lambda_{rls})= \max_{\|\y\|\le \delta}
\|\mf(\x^\lambda_{rls})-\mc^\top\y\|^2
\]
for $\delta\ge 0$.

 Under Assumption \ref{ass:qr}, the maximization
problem $\max_{\|\y\|\le \delta} \|\mf(\vx)-\mc^\top\y\|^2$
has a closed form solution
\[
\y_i =\delta  \left\{
\begin{array}{ll}
-1 & \mbox{if $(\mq^\top \mf(\vx))_i \ge 0$}\\
1 & \mbox{otherwise}
\end{array}
\right. \quad i=1,\ldots,r,
\]
where $\mq$ is the orthogonal matrix in the QR decomposition of $\mc=\mq\mr.$
 Moreover, by Theorem \ref{thm2.1}, we have
\[
Er_\delta(\x_{ls})=
\|\F(\x_{ls})\|^2+2\delta \| \mc^\top \mf (\x_{ls})\|_1+
\| \mc \|_F^2 \delta^2
\]
and
\[
Er_\delta(\x^\lambda_{rls})=
\|\F(\x^\lambda_{rls})\|^2+
2\delta \| \mc^\top \mf (\x^\lambda_{rls})\|_1+
\| \mc \|_F^2 \delta^2,
\]
which implies that  the worse-case residual errors of a least squares solution  $\x_{ls}$ and a $\lambda$-robust least squares solution $\x^\lambda_{rls}$ have the following difference
\begeq
\label{eq:slopeform}
\begin{array}{ll}
\Delta_\lambda(\delta) & =Er_\delta(\x_{ls})-Er_\delta(\x^\lambda_{rls})=
\|\F(\x_{ls})\|^2-\|\F(\x^\lambda_{rls})\|^2  \\
\\
&+ 2\delta (\| \mc^\top \mf (\x_{ls})\|_1
-\| \mc^\top \mf (\x^\lambda_{rls})\|_1).
\end{array}
\endeq
Hence, we have the following Proposition.
\begin{proposition}\label{rls-ls}
If $\x_{ls}$ is not a $\lambda$-robust least squares solution,
then $\Delta_\lambda(\delta)$ increases linearly with respect
to $\delta$ by the rate
$2(\| \mc^\top \mf (\x_{ls})\|_1
-\| \mc^\top \mf (\x^\lambda_{rls})\|_1),$ and $\Delta_\lambda(\delta)>0$ when
\begin{equation}\label{Delta}
\delta > \frac{\|\F(\x^\lambda_{rls})\|^2-\|\F(\x_{ls})\|^2}%
{2(\| \mc^\top \mf (\x_{ls})\|_1
-\| \mc^\top \mf (\x^\lambda_{rls})\|_1)}.
\end{equation}
\end{proposition}

\begin{proof}
If $\x_{ls}$ is not a minimizer of
$\|\mf(\x)\|^2+2\lambda \| \mc^\top \mf (\x)\|_1$, then
\[
0\ge \|\F(\x_{ls})\|^2-\|\F(\x^\lambda_{rls})\|^2 > 2\delta (\| \mc^\top \mf (\x^\lambda_{rls})\|_1
-\| \mc^\top \mf (\x_{ls})\|_1),
\]
where we use the definition of least square solution $\x_{ls}$ and
$\lambda-$robust least square solution $\x^\lambda_{rls}$. Hence we have
$ \| \mc^\top \mf (\x_{ls})\|_1
-\| \mc^\top \mf (\x^\lambda_{rls})\|_1>0$,
and thus $\Delta_\lambda(\delta)$ increases linearly with
respect to $\delta$ by the rate
$2(\| \mc^\top \mf (\x_{ls})\|_1
-\| \mc^\top \mf (\x^\lambda_{rls})\|_1)$
and $\Delta_\lambda(\delta)>0$  when (\ref{Delta})  holds.
\end{proof}

\subsection{Simple example}
\label{subsec:example51}

To illustrate Proposition \ref{rls-ls}, we consider a simple example.

\noindent
{\bf Example 5.1} Let
\[
\mf(\x)=\ma\x-\b, \quad \ma=[1, 1, 1]^\top,
\quad \vb=[1, 2, 4]^\top, \quad \mc=\mi.
\]

By a simple calculation, we find that  $\x_{ls}=7/3$,
and $\x^1_{rls}=2$, $\x^{0.2}_{rls}=2.2$ with $\lambda=1$
and $\lambda=0.2$, respectively. Moreover,  we have
\[
\frac{\|\F(\x^1_{rls})\|^2-\|\F(\x_{ls})\|^2}{2(\|\mf (\x_{ls})\|_1
-\|\mf (\x^1_{rls})\|_1)}=\frac{5-42/9}{2(10/3-3)}=\frac{1}{2}, \quad \quad
\quad \Delta_1(\delta)=\frac{2}{3}\delta-\frac{1}{3}
\]
and
\[
\frac{\|\F(\x^{0.2}_{rls})\|^2-\|\F(\x_{ls})\|^2}{2(\|\mf (\x_{ls})\|_1
-\|\mf (\x^{0.2}_{rls})\|_1)}=\frac{4.72-42/9}{2(10/3-3.2)}=0.2, \quad
\quad  \quad \Delta_{0.2}(\delta)=\frac{0.8}{3}\delta-\frac{0.16}{3}.
\]

It is interesting to observe that
\[
0\in \ma^\top(\ma \x^1_{rls} -\vb) +2\lambda \partial \|\ma \x^1_{rls}-\vb\|_1
\]
for any $\lambda\ge 1.$ This means that
$\x^\lambda_{rls}=\x^1_{rls}=2$ for all $\lambda \ge 1.$
In fact, for robust linear least squares problem, if  $\x^*$ is the minimizer of $\|\mc^\top \F(\x^*)\|_1$ and $0\in {\rm int} \partial
\|\mc^\top \F(\x^*)\|_1,$  then there is $\bar{\lambda}>0$ such that
$\x^*$ is a $\lambda$-least squares solution for
all $\lambda \ge \bar{\lambda}$.
 See (iv) of Remark 2.1.

\subsection{Integral equations with uncertain data}
\label{subsec:nlinteg}

We consider the nonlinear integral equation
\begeq
\label{eq:nleq}
\calg (u)(x) \equiv \int_0^1 g(x,y) \phi(u(y)) \dy = f(x),
\endeq
where $g$ is the
Greens function
\[
g(x,y) =
    \left\{\begin{array}{c}
        y (1-x), \quad {\rm if} \,\, \ x > y\\
        x (1-y), \quad {\rm otherwise.}
    \end{array}\right.
\]
The operator $G$
\begeq
\label{eq:gdeff}
G u(x) \equiv \int_0^1 g(x,y) u(y) \dy
\endeq
is the inverse of the negative second derivative
operator with homogeneous Dirichlet boundary conditions. So if $w = G v$, then
\[
-w'' = v, \quad w(0) = w(1) = 0.
\]

In our example
\begeq
\label{eq:nlphi}
\phi(u) = \frac{\sin(\pi u) + u^3}{1 + u^2}.
\endeq
We define
\[
f(x) = G \phi(u^*(x)) \quad {\rm with} \quad   u^*(x)=|x-0.025|
\]
for the example.
We add an $H^1$ norm regularization to stabilize the problem, so the
minimization problem becomes
\[
%\min \| \mg_h ( \phi ( \mt \vp ) ) - \vf_h \|_2^2  +
\min \| G(\phi(u)) - f \|_{L_2}^2 +
\alpha \| d u /dx \|_2^2.
\]
We restrict $u$ to be in
the span of a low dimensional space by
\[
u(x) = \sum_{i=1}^{n} p_j T_{i-1}(x) \equiv \calt \vp
\]
where $T_i$ is the $i$th degree Chebyshev polynomial on $[0,1]$.
So, letting $\calt$ be the projection to the span of the first $n$
Chebyshev polynomials our infinite dimensional problem is to minimize
the $L^2$ norm of
\begeq
\label{eq:infdim}
\calf(\vp) = \left(
\begin{array}{l}
G(\phi({\calt} \vp)) - f \\
\sqrt{\alpha} d({\calt} \vp)/dx
\end{array}
\right)
\endeq
over $\vp \in R^n$.

We discretize the integral operator $G$ in \eqnok{nleq} with $\vu$ by the
composite trapezoid rule on a grid with $m_x$ points to obtain a
discrete operator $\mg_h$.
The discrete operator becomes more ill-conditioned as the number of
quadrature points increases, but remains uniformly bounded in
the operator norm. Similarly we let
$\mt$ be the $m_x \times n$ matrix
whose the columns are the Chebyshev
polynomials evaluated at the grid points.
In our computations we let
$m_x = 1000$ and $n=10$.

For the discrete problem (\ref{nls}) we see that
$\mf$ in the least squares residual (\ref{nls}) becomes
\[
\mf(\vp) = \left( \begin{array}{c}
\mg_h ( \phi ( \mt \vp ) ) - \vf_h \\
\sqrt{\alpha} D_x \mt \vp
\end{array}
\right)
\]
Hence $m = 2 m_x$. In our computations we approximate the first derivative
with a simple one-sided difference $D_x$. 
We set $m_x = 1000$ and so $m = 2 m_x = 2000$.
Finally, we note that only the original least squares residual
$G(\phi({\calt} \vp)) - f$ has potential errors in the data. Hence,
we apply the robust correction only to that part of $\calf$. This has
the effect of setting
\[
\mc = [\mi_{m_x \times m_x}, 0_{m_x \times m_x} ]^\top
\]
with $r = m_x$.
We used $n=10$, $\mu = 10^{-8}$, and $\alpha = .01$
in the computations.

We use the Optimization.jl \citep{optimizationjl}
Julia \citep{Juliasirev} package to
interface to a BFGS optimization
for the minimization problem

\begeq
\label{eq:sminner}
\min_\vp \Psi_\lambda(\vp,\mu):= \|\mf(\vp)\|^2+2\lambda
\Theta(\mc^\top \mf(\vp), \mu),
\endeq
where we approximate  $\|\mc^\top \F(\vp)\|_1$
using the smoothing function
$\theta(t,\mu)=\sqrt{t^2+4\mu^2} $ of $t$ as in  \eqnok{Thetadef}.
In our computation, we used the Zygote \citep{zygote}
automatic differentiation package
to compute gradients.

In Table~\ref{tab:tab1} we show for several values of $\lambda$ the
function values and gradient norms upon termination of the optimizer.
We also report the computed value of the slope 
\[{\rm slope}(\vp_{ls},\vp^\lambda_{rls})=
2 (\| \mc^\top \mf (\vp_{ls})\|_1 -\| \mc^\top \mf (\vp^\lambda_{rls})\|_1)
\]
to show how the slope rapidly approaches a constant value as
$\lambda$ increases, which verifies the theoretical results of Proposition 5.1 with $\x$ being replaced by $\vp$ in \eqnok{slopeform}. This is also illustrated by Figure~\ref{fig:robust01}.

{\color{red}
\begin{table}[h!]
\caption{\label{tab:tab1} Values of $\Psi_\lambda$ and
$\| \nabla \Psi_\lambda \|$ at computed minimizers}
\centerline{
\begin{tabular}{cccc}
$\lambda$ &$\Psi_\lambda$ &$\| \nabla \Psi_\lambda \|$ &    slope$(\vp_{ls},\vp^\lambda_{rls})$\\
\hline
0.00e+00 & 2.21e-02 & 1.14e-12 & 0.00e+00   \\
1.00e-01 & 8.30e-01 & 2.27e-12 & 7.27e-01   \\
1.00e+00 & 5.05e+00 & 2.49e-11 & 6.15e+00   \\
5.00e+00 & 7.72e+00 & 2.44e-09 & 8.16e+00   \\
1.00e+01 & 8.85e+00 & 6.24e-10 & 8.27e+00   \\
1.00e+02 & 1.25e+01 & 5.69e-08 & 8.44e+00   \\
\hline
\end{tabular}
}
\end{table}
}

In this example, we see that if we set $
{\hat \epsilon} (\lambda) =  8 r \lambda \mu  = 8 \lambda 10^{-5}
$, then
$
{\hat \epsilon} (\lambda) \le {\hat \epsilon}^* = 8 \times 10^{-5}
$
for any $0 <\lambda \le 1$. Hence
we have an ${\hat \epsilon}^*$ minimax point for any $0 < \lambda \le 1$.

We follow the procedure at the beginning of \S 5 to compute the worst-case residual errors
\[
Er_\delta(\vp) =  \| \mf (\vp) \|^2_2 +
2 \delta \| \mc^\top \mf(\vp) \|_1 + (m \delta)^2
\]
at the least square solution $\vp_{ls}$ and the robust least squares solution
$\vp^\lambda_{rls}$.

In Figure~\ref{fig:robust01}, we
plot $\Delta_\lambda(\delta)=Er_\delta(\vp_{ls})  - Er_\delta(\vp_{rls}^\lambda)$ as a function
of $\delta$ and five values of $\lambda$.
Figure~\ref{fig:robust01} (a) focuses on small values of $\delta$ to show the
behavior near $\delta = 0$ and Figure~\ref{fig:robust01} (b) illustrates the
behavior for larger $\delta$. The slopes for the larger values of
$\lambda$ are indistinguishable, as both the theory and
Table~\ref{tab:tab1} predict.

\begin{figure}[h]
\begin{center}
	\subfigure[Difference in averages with small $\delta$]{
	%	\begin{minipage}[b]{0.41\textwidth}
			\includegraphics[width=0.5\textwidth]{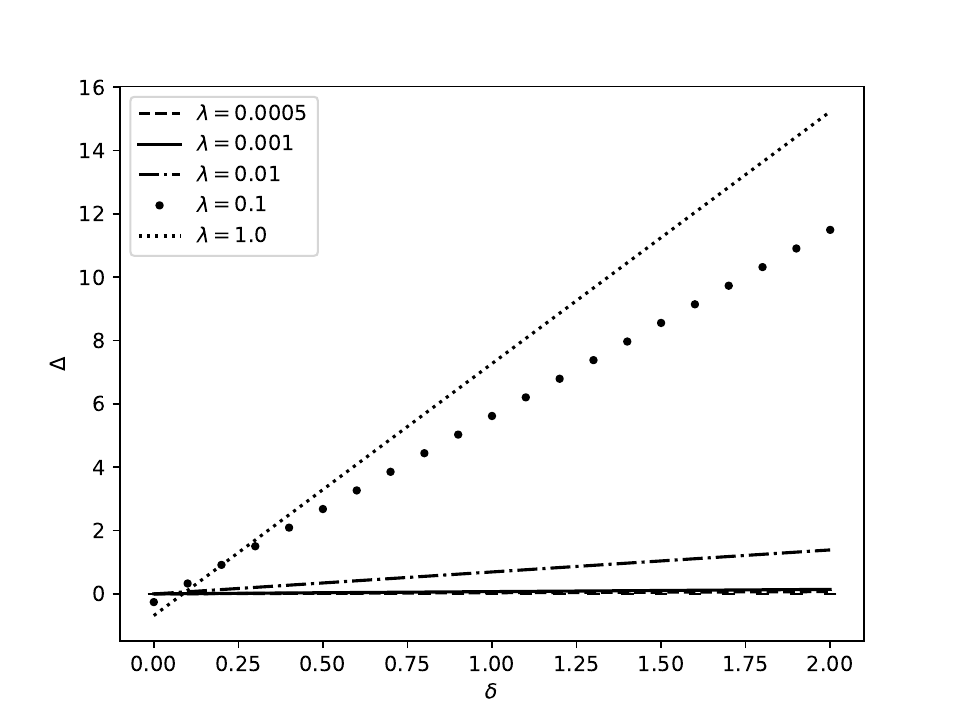}\hspace{-5mm}
%\caption{\label{fig:robust} Difference in Averages}
%		\end{minipage}
	}\label{fig:robust}
\hspace{0.05in}
	\subfigure[Difference in averages with large $\delta$]{
%		\begin{minipage}[b]{0.41\textwidth}
			\includegraphics[width=0.5\textwidth]{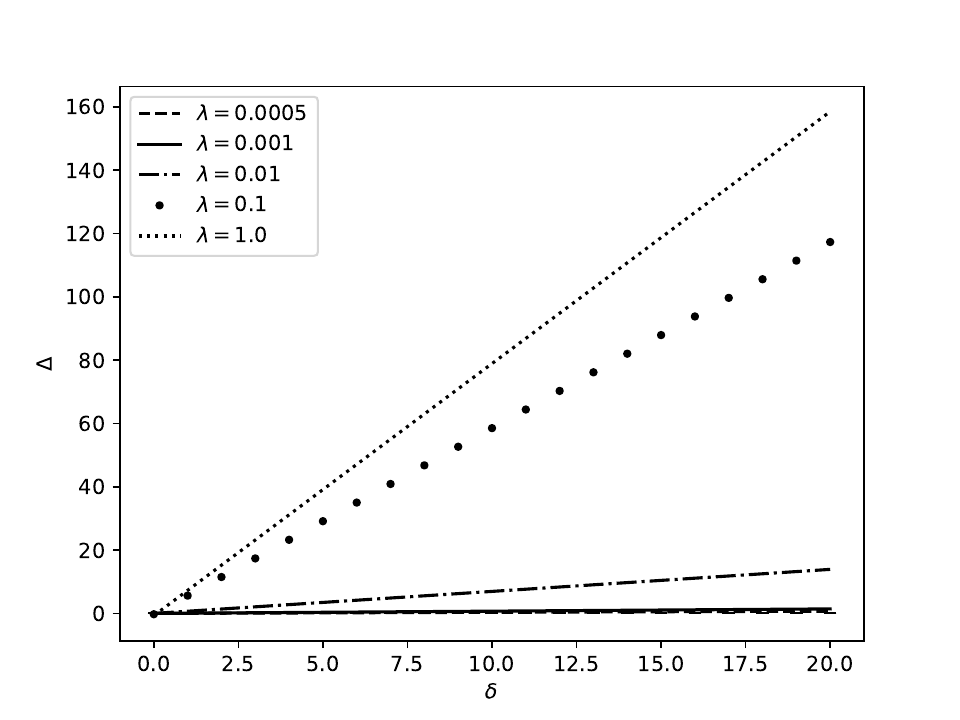}\hspace{-5mm}}\label{fig:robustl}
%\caption{\label{fig:robustl} Difference in Averages}
%		\end{minipage}	}\hspace{-5mm}
\end{center}
\caption{\label{fig:robust01} Nonlinear linear least squares example, $\Delta_\lambda(\delta)=Er_\delta(\vp_{ls})  - Er_\delta(\vp_{rls}^\lambda)$ }
\end{figure}

\vspace{-0.15in}

\subsection{Linear least squares example}
\label{subsec:lsqeg}
In this section we consider a linear form of the example in
\S~\ref{subsec:nlinteg}. The only difference is that we use
\begeq
\label{eq:linphi}
\phi(u) =\pi u
\endeq
in place of \eqnok{nlphi}. Its discrete problem becomes linear least squares problem
with
\[
\mf(\vp) = \left( \begin{array}{c}
\pi\mg_h \mt \vp   - \vf_h \\
\pi\sqrt{\alpha} D_x \mt \vp
\end{array}
\right).
\] The least square solution is
$$\vp_{ls}=(\ma^\top \ma)^{-1}\ma^\top \vb,  \quad   {\rm where} \quad
\ma=\left( \begin{array}{c}
\pi\mg_h \mt \\
\pi\sqrt{\alpha} D_x \mt
\end{array}
\right),     \,\, \vb=\left( \begin{array}{c}
\vf_h \\
0
\end{array}
\right).
$$
The robust squares solution is
\begin{equation}\label{RLLS}
\vp_{rls}^\lambda=\mathrm{argmin} \Psi_\lambda(\vp,\mu):= \|\ma \vp-\vb\|^2+2\lambda
\Theta(\mc^\top (\ma \vp-\vb), \mu).
\end{equation}
We present numerical results in the first four columns of Table 2 and plot
$\Delta_\lambda(\delta)=Er_\delta(\vp_{ls})  - Er_\delta(\vp_{rls}^\lambda)$ as a function
of $\delta$ and five values of $\lambda$
in Figure 2 as we
did in \S~\ref{subsec:nlinteg} and the conclusions are the same.

\begin{table}[h!]
\caption{\label{tab:tab2} Linear Case: Values of $\Psi_\lambda$ and
$\| \nabla \Psi_\lambda \|$ at computed minimizers}
\centerline{
\begin{tabular}{ccccc}
$\lambda$ &$\Psi_\lambda$ &$\| \nabla \Psi_\lambda \|$ &    slope$(\vp_{ls},\vp^\lambda_{rls})$ & slope$(\vp^\lambda_{las},\vp^\lambda_{rls})$ \\
\hline
0.00e+00 & 4.51e-02 & 9.76e-16 & 0.00e+00 & 0.00e+00\\
1.00e+00 & 8.08e+00 & 2.91e-08 & 7.95e+00 & 9.01e+00 \\
5.00e-01 & 5.05e+00 & 1.45e-10 & 3.98e+00 & 4.36e+00   \\
1.00e+01 & 1.80e+01 & 2.60e-08 & 1.15e+01 &6.30e+01 \\
1.00e+02 & 3.62e+01 & 5.54e-08 & 1.19e+01 &9.86e+01 \\
2.00e+02 & 4.06e+01 & 5.23e-08 & 1.20e+01 &9.86e+01 \\
\hline
\end{tabular}
}
\end{table}

\begin{figure}[h]
\begin{center}
        \subfigure[Difference in averages with small $\delta$]{
 %               \begin{minipage}[b]{0.41\textwidth}
                        \includegraphics[width=0.5\textwidth]{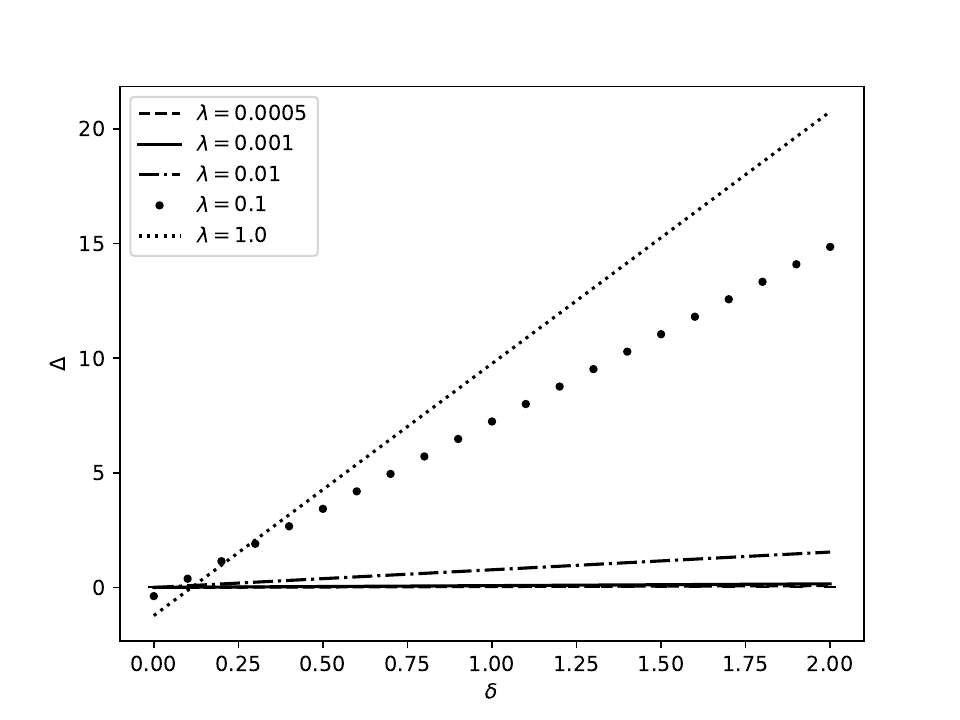}\hspace{-5mm}  }
 %                                             \includegraphics[width=1.1\textwidth]{robust3.pdf}\hspace{-5mm}
%\caption{\label{fig:robustlina} Linear Case: Difference in Averages}
%                \end{minipage}}
        \hspace{0.05in}
        \subfigure[Difference in averages with large $\delta$]{
 %               \begin{minipage}[b]{0.41\textwidth}
                       \includegraphics[width=0.5\textwidth]{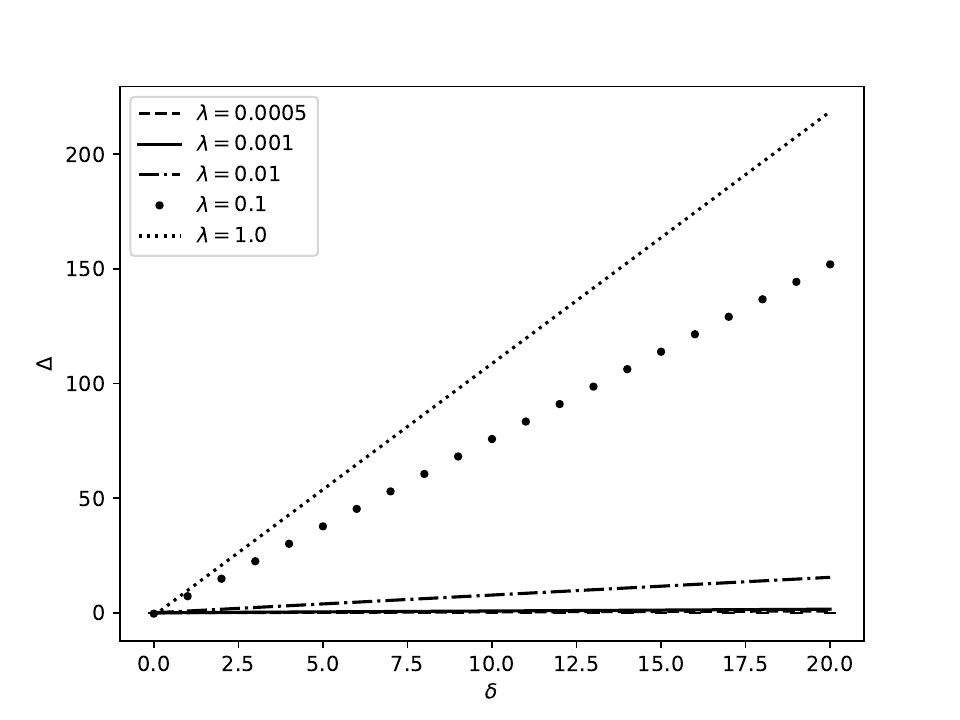}\hspace{-5mm}  }
%\caption{\label{fig:robustlinb} Linear Case: Difference in Averages}
%                \end{minipage}  }\hspace{-5mm}
\end{center}
\caption{Linear least squares example, $\Delta_\lambda(\delta)=Er_\delta(\vp_{ls})  - Er_\delta(\vp_{rls}^\lambda)$} \label{Figure2}
\end{figure}

%\clearpage

\subsection{Robust linear least squares solutions vs Lasso solutions in the worst-case residual errors}

We continue to consider the linear least squares problem from \S~\ref{subsec:lsqeg} and
compare the robust solution with the Lasso solution.
Lasso  (least absolute shrinkage and selection operator)
is a popular regression analysis model introduced by  \citet{Lasso}, which
is a $\ell_1$   regularized  linear least square problem
\begin{equation}\label{Lasso}
\min \|\ma \vp-\vb\|^2 +  \tau \|\vp\|_1.
\end{equation}
Adding the $\ell_1$ regularization term is to improve the prediction accuracy and robustness of the solution.

In the context of this paper, we use $\tau = 2 \lambda$ to compare
the solutions $\vp^\lambda_{las}$ of (\ref{Lasso})
to the robust least squares solution $\vp_{rls}^\lambda$ of
(\ref{RLLS})
for several values of $\lambda$. We follow the methods of the previous sections
and look at the worst-case residual errors $Er_\delta(\vp)$ for the robust least squares solution
$\vp_{rls}^\lambda$ and the Lasso solution $\vp_{las}^\lambda$.
We present the slope in the last column of Table~\ref{tab:tab2}.
In Figure 3, we
plot $\Delta_\lambda(\delta)=Er_\delta(\vp_{las}^\lambda)  - Er_\delta(\vp_{rls}^\lambda)$ as a function
of $\delta$ and five values of $\lambda$.

\begin{figure}[h]
\begin{center}
        \subfigure[Difference in averages with small $\delta$]{
%                \begin{minipage}[b]{0.41\textwidth}
                        \includegraphics[width=0.45\textwidth]{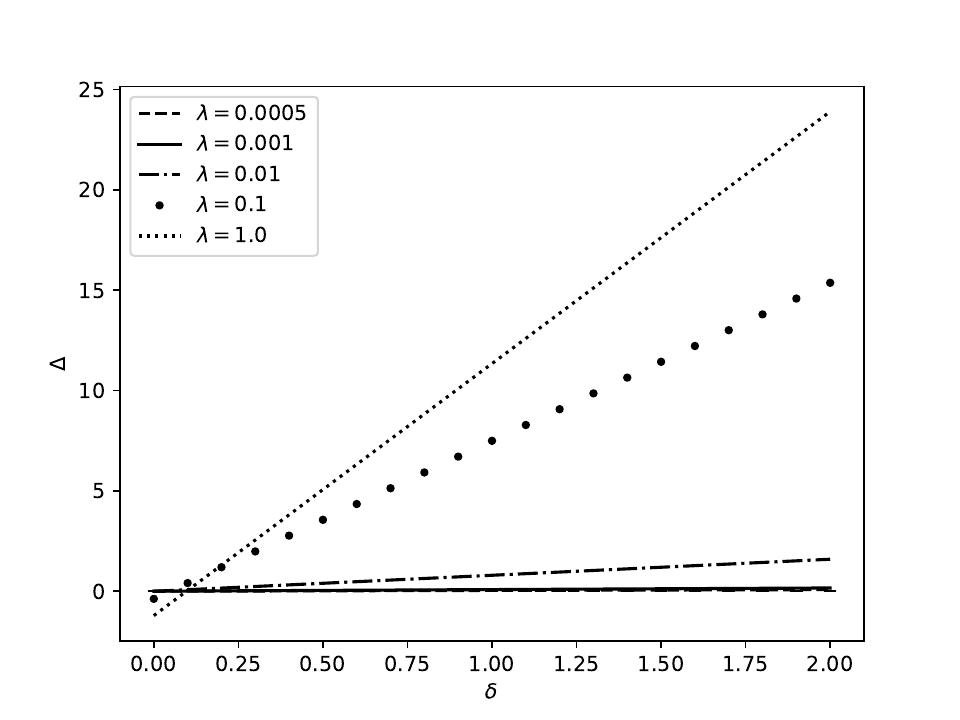}
%\caption{\label{fig:robustlasso} Lasso Case: Difference in Averages}
%                \end{minipage}
        }
\hspace{0.05in}
        \subfigure[Difference in averages with lager $\delta$]{
 %               \begin{minipage}[b]{0.41\textwidth}
                        \includegraphics[width=0.45\textwidth]{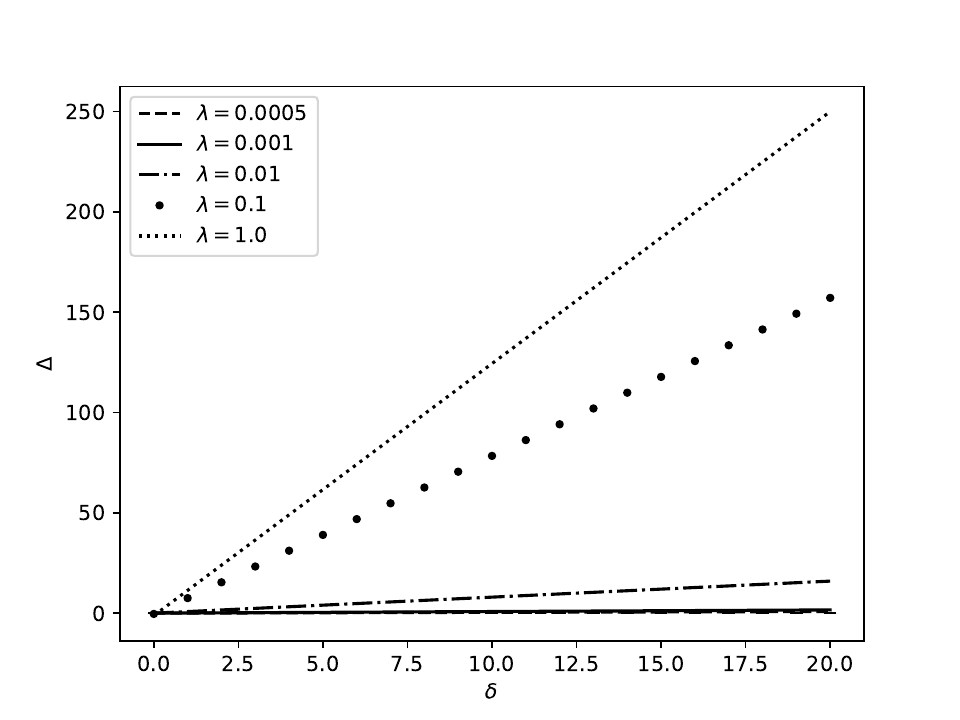} }
%\caption{\label{fig:robustlassob} Lasso Case: Difference in Averages}
%                \end{minipage}  }\hspace{-5mm}
\end{center}
\caption{Robust linear least squares vs Lasso,  $\Delta_\lambda(\delta)=Er_\delta(\vp^\lambda_{las})  - Er_\delta(\vp_{rls}^\lambda)$}
\end{figure}

\subsubsection{Reproducibility}
\label{subsec:codes}

We store the Julia codes that
generated the table and figures in the GitHub repository:
\url{https://github.com/ctkelley/RLS_Example.jl}

{\color{black}
\section{Conclusion}
We consider the min-max problem (\ref{eq2}) for robust nonlinear
least squares problems. We use the example in subsection 1.2 to show that  the min-max problem (\ref{eq2}) does not have a local minimax point and the first order optimality
conditions defined by gradients of the objective function  can lead to incorrect solutions.  We give an explicit formula (\ref{inner})
for the value function of the inner maximization problem in (\ref{eq2}). We provide error bounds from  a solution of
the least squares problem (\ref{nls}) to the solution set of the
robust least squares problem.  Moreover, we propose a smoothing method for finding a
minimax point of the min-max problem (\ref{eq2}) by using the formula (\ref{inner}).
We show that finding an $\epsilon$ minimax critical point of the min-max problem (\ref{eq2})
needs at most $O(\epsilon^{-2} +\delta^2\epsilon^{-3}) $
evaluations of the function value and gradient of the
objective function by the smoothing method. Moreover, we give an explicit form
    for the difference $\Delta(\delta)$ of
    worst-case residual errors of a least squares solution and a robust least squares solution, and show  $\Delta(\delta)$ increases linearly as $\delta$ increases. Numerical results of  integral equations with uncertain data demonstrate the robustness of solutions of our approach and unstable behaviour of
least squares solutions disregarding uncertainties in the data.

\section{Funding}

We would like to acknowledge support for this project from
National Science Foundation Grant
DMS-1906446 and
Hong Kong Research Grant Council PolyU15300123.

%\section{Acknowledgement}
%We would like to thank the Associate Editor and two reviewers for their helpful comments and
\bibliographystyle{IMANUM-BIB}
\bibliography{MaxProb}

}

\end{document}